\documentclass[a4paper,11pt]{amsart}
\usepackage[left=2.5cm,right=2.5cm,top=2.7cm,bottom=2.7cm]{geometry}
\newcommand*{\mailto}[1]{\href{mailto:#1}{\nolinkurl{#1}}}

\usepackage{graphicx}
\usepackage{color}
\usepackage{epstopdf}
\usepackage{pdfsync}
\usepackage{amssymb,amsfonts,amsthm,amssymb}
\usepackage{amsmath}
\usepackage{mathrsfs}
\usepackage{calc}
\usepackage{graphics}

\newcommand{\vep}{\varepsilon}

\def\bar{\overline}
\def\ul{\underline}
\def\ol{\overline}
\def\bu{\mathbf{u}}
\def\bv{\mathbf{v}}
\def\cB{\mathcal{B}}
\newtheorem{theorem}{Theorem}[section]
\theoremstyle{remark}
\newtheorem{remark}[theorem]{Remark}
\newtheorem{corollary}[theorem]{Corollary}

\theoremstyle{proposition}
\newtheorem{proposition}[theorem]{Proposition}

\numberwithin{equation}{section}

\def\pp{p}
\def\qq{q}
\def\R{\mathbb{R}}
\def\N{\mathbb{N}}
\def\cT{\mathcal{T}}
\def\cW{\mathcal{W}}
\def\psii{u}

\date{\today}
\begin{document}
\title[Westervelt]{Relaxation of regularity for the Westervelt equation by nonlinear damping with application in acoustic-acoustic and elastic-acoustic coupling}

\author[R.\ Brunnhuber]{Rainer Brunnhuber}
\address{Institut f\"ur Mathematik\\ Alpen-Adria-Universit\"at Klagenfurt\\
Universit\"atsstra{\ss}e 65-57\\ 9020 Klagenfurt am W\"orthersee\\ Austria}
\email{rainer.brunnhuber@aau.at}
\urladdr{http://www.aau.at/~rabrunnh/}

\author[B.\ Kaltenbacher]{Barbara Kaltenbacher}
\address{Institut f\"ur Mathematik\\ Alpen-Adria-Universit\"at Klagenfurt\\
Universit\"atsstra{\ss}e 65-57\\ 9020 Klagenfurt am W\"orthersee\\ Austria}
\email{barbara.kaltenbacher@uni-klu.at}
\urladdr{http://www.aau.at/~bkaltenb/}

\author[P.\ Radu]{Petronela Radu}
\address{Department of Mathematics\\
University of Nebraska - Lincoln\\
Avery Hall 239\\ Lincoln, NE 68588\\ United States of America}
\email{pradu@math.unl.edu}
\urladdr{http://http://www.math.unl.edu/~pradu3/}

\begin{abstract}
 In this paper we show local (and partially global) in time existence for the Westervelt equation with several
versions of nonlinear damping. This enables us to prove well-posedness with spatially varying $L_\infty$-coefficients,
which includes the situation of interface coupling between linear and nonlinear acoustics as well as between linear
elasticity and nonlinear acoustics, as relevant, e.g., in high intensity focused ultrasound (HIFU) applications.
\end{abstract}

\maketitle


\section{Introduction}
The Westervelt equation for the acoustic pressure $p$
\begin{equation}\label{Westervelt}
(p-k(p^2))_{tt} -c^2 \Delta p - b\Delta p_t = 0,
\end{equation}
where $k = \beta_a / \lambda$, $\beta_a = 1 + B/(2A)$,
$\lambda=\varrho c^2$ is the bulk modulus,
$\varrho$ is the mass density,
$c$ is the speed of sound,
$b$ is the diffusivity of sound,
$B/A$ is the parameter of nonlinearity,
is a classical model of nonlinear acoustics,
cf.~\cite{HamiltonBlackstock,manfred,Westervelt}.
In particular, it is widely used for the simulation of high intensity focused ultrasound (HIFU) which
has a broad range of technical and medical applications ranging from lithotripsy or thermotherapy
to ultrasound cleaning or welding and sonochemistry, see~\cite{abramov,manfred} and the references therein.
Note that the Westervelt equation can be written alternatively in terms of the acoustic velocity potential $\psi$,
\begin{equation}\label{Westervelt_psi}
(\psi_t-\tilde{k}(\psi_t^2))_t -c^2 \Delta \psi - b\Delta\psi_t
= 0,
\end{equation}
where $\varrho \psi_t = p$ and $\tilde{k}=\varrho k$.

An analysis of the Westervelt equation with homogeneous~\cite{KL08} and inhomogeneous~\cite{KLV10} Dirichlet
and Neumann~\cite{Neumann} boundary conditions as well as with boundary instead of interior damping~\cite{K10}
has yielded well-posedness and exponential decay of small and regular (i.e., $H^2(\Omega)$) solutions.
Here $\Omega\subseteq\R^d$, $d\in\{1,2,3\}$ is the spatial domain on which \eqref{Westervelt} is considered.

An important open question to be addressed in this paper is existence of {\em spatially less regular solutions}
as needed, e.g., for coupling with elastic or acoustic regions exhibiting different material parameters in the
simulation of a focusing silicone lens immersed in an acoustic medium.

An important feature of equation \eqref{Westervelt} (similarly also of \eqref{Westervelt_psi}) is the potential
degeneracy due to the factor $(1-2kp)$ of the second time derivative $p_{tt}$. In order to avoid degeneracy,
it is crucial to obtain an $L_\infty$-estimate of $p$ in order to bound $(1-2kp)$ away from zero.
This has been achieved until now by combining a bound of $\Delta p$ (obtained by energy estimates) with
Sobolev's embedding $H^2(\Omega)\to L_\infty(\Omega)$. In this paper, we will employ the nonlinear damping
for this purpose instead. Note that, indeed, the particular choice of the damping is to some extent left open
from the point of view of the physical modeling. We will use this freedom to devise possible damping terms
leading to existence of $H^1(\Omega)$-solutions.

We consider three modifications of the homogeneous Dirichlet boundary value problem associated with
\eqref{Westervelt} and \eqref{Westervelt_psi}, namely
\begin{equation}\label{Wpress_pLaplace}
\begin{cases}
(1-2ku)u_{tt}
-c^2\,\text{div}\,\Bigl(\nabla u+\vep|\nabla u|^{\pp-1}\nabla u\Bigr)-b\Delta u_t=2k(u_t)^2,
\\
(u,u_t)|_{t=0}=(u_0, u_1),
\\
u|_{\partial \Omega} =0,
\end{cases}
\end{equation}
\begin{equation}\label{Wpress_viscosity}
\begin{cases}
(1-2ku)u_{tt}-c^2\Delta u-b\,\text{div}\,\Bigl(((1-\delta) +\delta|\nabla u_t|^{\qq-1})\nabla u_t\Bigr)
=2k(u_t)^2,
\\
(u,u_t)|_{t=0}=(u_0, u_1),
\\
u|_{\partial \Omega} =0,
.\\
\end{cases}
\end{equation}
\begin{equation}\label{Wpot_viscosity}
\begin{cases}
\psii_{tt}-\frac{c^2}{1-2\tilde{k}\psii_t}\Delta \psii
-b\,\text{div}\,\Bigl(((1-\delta) +\delta|\nabla \psii_t|^{\qq-1})\nabla \psii_t\Bigr)
=0,
\\
(\psii,\psii_t)|_{t=0}=(\psii_0, \psii_1),
\\
\psii|_{\partial \Omega} =0,
.\\
\end{cases}
\end{equation}
in a smooth domain $\Omega\subseteq\R^d$, $d\in\{1,2,3\}$ with $\vep\geq0$, $\delta\in [0,1]$ and $\pp,\qq\geq1$.
The constant parameters $b,c$ will be positive while $k$ will not be assumed to have a particular sign.
Equations \eqref{Wpress_pLaplace} and \eqref{Wpress_viscosity} are motivated by \eqref{Westervelt}
 while  \eqref{Wpot_viscosity} comes from \eqref{Westervelt_psi}.
 In all equations we changed the notation to the typical mathematical one for solutions of PDEs, i.e., $p\to u$, and $\psi\to\psii$, respectively.

Note that if $\vep=0$ or if $\pp=1$ in \eqref{Wpress_pLaplace}, one obtains the classical Westervelt equation;
likewise, for $\delta=0$ or $\qq=1$ in \eqref{Wpress_viscosity} and \eqref{Wpot_viscosity}. Here we will analyze the problem for
\[
\vep,\delta>0\text{ and }\pp,\qq>d-1\,,
\]
where $d$ is the space dimension
such that $W^{1,\pp+1}(\Omega)$ and $W^{1,\qq+1}(\Omega)$ are continuously embedded in $L_\infty(\Omega)$; precise conditions for these inclusions to hold will be discussed later in the paper. In Sections \ref{secWpot_viscosity} and \ref{secWelastic_viscosity} we will have to to use the stronger assumption $\qq\geq 3$,
since there we will need $W^{1,\qq+1}(\Omega)$ to be continuosly embedded in $W^{1,4}(\Omega)$.

In all three cases the damping terms enable us to derive an $L_\infty$-estimate on $u$ (or $u_t$)
and thus avoid degeneracy of the coefficient $1-2ku$ (or $1-2ku_t$). Namely, for the damping term
of \eqref{Wpress_pLaplace} we have, using homogeneity of the Dirichlet boundary values of $u$,
\begin{equation}
\label{estLinftyW1}
\begin{aligned}
&|u(t,x)|\leq C_{W_0^{1,\pp+1},L_\infty}^\Omega
|\nabla u(t)|_{L_{\pp+1}}\\
&\qquad =C_{W_0^{1,\pp+1},L_\infty}^\Omega
\Bigl(|\nabla u_0|_{L_{\pp+1}}^{\pp+1}
+\int_0^t\frac{d}{dt}\int_\Omega|\nabla u(s,y)|^{\pp+1}\,dy\,ds
\Bigr)^{\frac{1}{\pp+1}}\\
&\qquad =C_{W_0^{1,\pp+1},L_\infty}^\Omega\Bigl[
|\nabla u_0|_{L_{\pp+1}}^{\pp+1}\\
&\qquad \quad +
\int_0^t\int_\Omega (\pp+1)|\nabla u(s,y)|^{\pp-1}\nabla u(s,y) \nabla u_t(s,y) \, dy\, ds \Bigr]^{\frac{1}{\pp+1}} \\
&\qquad =C_{W_0^{1,\pp+1},L_\infty}^\Omega\Bigl[
|\nabla u_0|_{L_{\pp+1}}^{\pp+1}\\
& \qquad \quad
-(\pp+1)\int_0^t\int_\Omega \text{div}\,\Bigl(|\nabla u(s,y)|^{\pp-1}\nabla u(s,y)\Bigr) u_t(s,y) \, dy\, ds \Bigr]^{\frac{1}{\pp+1}}, \\
\end{aligned}
\end{equation}
where $C_{W_0^{1,\pp+1},L_\infty}^\Omega$ denotes a combination of the constant in the Sobolev embedding
$W_0^{1,\pp+1}(\Omega)\to L_\infty(\Omega)$ with the one from the Poincar\'{e}-Friedrichs inequality.\\
For the damping term of \eqref{Wpress_viscosity}, we have
\begin{equation}
\label{estLinftyW2}
\begin{aligned}
&|u(t,x)|\leq C_{W_0^{1,\qq+1},L_\infty}^\Omega
|\nabla u(t)|_{L_{\qq+1}}\\
&\qquad=C_{W_0^{1,\qq+1},L_\infty}^\Omega
\Big|\nabla u_0+\int_0^t \nabla u_t(s)\,ds\Big|_{L_{\qq+1}(\Omega)} \\
&\qquad \leq C_{W_0^{1,\qq+1},L_\infty}^\Omega
\Bigl[|\nabla u_0|_{L_{\qq+1}}
+\Big|\int_0^t\nabla u_t(s)\,ds\Big|_{L_{\qq+1}}\Bigr]\\
&\qquad \leq C_{W_0^{1,\qq+1},L_\infty}^\Omega
\Bigl[|\nabla u_0|_{L_{\qq+1}}
+\Bigl(t^{\qq}\int_0^t\int_\Omega \Big|\nabla u_t(s,y)\Big|^{\qq+1} \, dy \,ds
\Bigr)^{\frac{1}{\qq+1}}\Bigr] \\
&\qquad =C_{W_0^{1,\qq+1},L_\infty}^\Omega
\Bigl[|\nabla u_0|_{L_{\qq+1}}\\
&\qquad \quad +\Bigl(-t^q\int_0^t\int_\Omega \text{div}\,\Bigl(|\nabla u_t(s,y)|^{\qq-1}\nabla u_t(s,y)\Bigr) u_t(s,y) \, dy\, ds \Bigr)^{\frac{1}{\qq+1}}\Bigr].
\end{aligned}
\end{equation}
Finally, for \eqref{Wpot_viscosity}, we get, replacing $u$ by $u_t$ and $\pp$ by $\qq$ in  \eqref{estLinftyW1}
\begin{equation}
\label{estLinftyW3}
\begin{aligned}
&|\psii_t(t,x)|\leq
C_{W_0^{1,\qq+1},L_\infty}^\Omega\Bigl[
|\nabla \psii_1|_{L_{\qq+1}}^{\qq+1}\\
 & \qquad -(\qq+1)\int_0^t\int_\Omega \text{div}\,\Bigl(|\nabla \psii_t(s,y)|^{\qq-1}\nabla \psii_t(s,y)\Bigr)
\psii_{tt}(s,y) \, dy\, ds\Bigr]^{\frac{1}{\qq+1}}. \\
\end{aligned}
\end{equation}
Hence, in all three cases, multiplication of the damping term with $u_t$ (or $\psii_{tt}$) and integration over time and space
will provide us with the desired $L_\infty$-estimate on $u$ (or $\psii_t$).
In this manner, we avoid estimates on $\Delta u$ in order to conclude $L_\infty$-boundedness of $u$,
a strategy that has been used in previous studies of the Westervelt equation~\cite{KL08,KLV10}.

A major advantage of this relaxed regularity is the fact that it can be expected to enable solutions
of the modified Westervelt equation to be coupled with other equations or with jumping coefficients,
a situation that is also of high practical relevance for HIFU devices based on the use of acoustic lenses
immersed in a fluid medium~\cite{manfred}.

We will first consider an acoustic-acoustic coupling. This can be modeled by the presence of spatially varying,
namely piecewise constant coefficients, in a pressure formulation (cf.~\cite{BambergerGlowinskiTran} for the linear case)
\begin{equation}\label{Wacoustic_viscosity}
\begin{cases}
\frac{1}{\lambda(x)}(1-2k(x)u)u_{tt}-\text{div}\,(\frac{1}{\varrho(x)}\nabla u)\\ -\text{div}\,\Bigl(b(x)((1-\delta(x)) +\delta(x)|\nabla u_t|^{\qq-1})\nabla u_t\Bigr)
=\frac{2k(x)}{\lambda(x)}(u_t)^2,
\\
(u,u_t)|_{t=0}=(u_0, u_1),
\\
u|_{\partial \Omega} =0
.\\
\end{cases}
\end{equation}
Here we have emphasized space dependence of the coefficients while suppressing space and
time dependence of $u$ in the notation as before.
While in the above equations \eqref{Wpress_pLaplace}, \eqref{Wpress_viscosity}, \eqref{Wpot_viscosity}
and in the regions of nonlinearity (i.e., $k\not=0$) in \eqref{Wacoustic_viscosity}, strong damping $b>0$ is needed for
ensuring well-posedness, we may set $b=0$ in regions where  $k$ vanishes.
This corresponds to the physically relevant situation of a linearly acoustic (possibly to be considered as approximation
to linearly elastic) silicone lens immersed in a nonlinear acoustic fluid.

The physically more relevant model requires a linearly elastic model for the lens. Therefore, we consider a velocity based formulation for elastic-acoustic coupling (see also the displacement based formulation in~\cite{BermudezRodriguezSantamarina} and the velocity potential formulation in~\cite{FlemischKaltenbacherWohlmuth}, both for the linear case)
\begin{equation}\label{Welastic_viscosity}
\begin{cases}
\varrho(x) \bu_{tt}-\cB^T\frac{1}{1-2\tilde{k}(x)\psi_t} [c](x)\cB\bu\\
\qquad +\cB^T\Bigl(((1-\delta(x)) +\delta(x)|\cB \bu_t|^{\qq-1})[b](x)\cB \bu_t\Bigr)
=0,
\\
(\bu,\bu_t)|_{t=0}=(u_0, u_1),
\\
\bu|_{\partial \Omega} =0
,\\
\end{cases}
\end{equation}
where $\mathbf{u}$ plays the role of the velocity, $\psi$ determines the gradient part in the Helmholtz decomposition of $\mathbf{u}$,
$$\mathbf{u}=\nabla\psi+\nabla\times\mathbf{A}\,,$$
and the first order differential operator $\cB$ is given by
$$ \cB=\left( \begin{array}{cccccc}
        \partial_{x_1} & 0 & 0 & 0 & \partial_{x_3} & \partial_{x_2}\\
        0 & \partial_{x_2} & 0 & \partial_{x_3} & 0 & \partial_{x_1}\\
        0 & 0 & \partial_{x_3} & \partial_{x_2} & \partial_{x_1} & 0
                      \end{array} \right)^T
.$$

Note that $\psi$ can be determined from $\bu$ as the solution of
\[
-\Delta \psi=-\text{div}\,\bu
\]
which is unique, e.g., if we imposed homogeneous Dirichlet boundary conditions on $\psi$, which we will do below.
Here we think of $\Omega$ being decomposed into an acoustic (fluid) and an elastic (solid) subdomain
$$\Omega=\Omega_f\cup\Omega_s\,, \quad \Omega_f\cap\Omega_s=\emptyset\,.$$

No particular smoothness assumption will have to be imposed on the domains $\Omega_f$ and $\Omega_s$ as these subdomains are just characterized as the sets of points where the $L_\infty$-coefficients $\varrho$, $c$, $\tilde{k}$, and $b$ take on certain values that are typical for fluid and solid, respectively.
The acoustic region $\Omega_f$ is characterized as the region of vanishing shear modulus $\mu=0$ in the tensor
$$[c]=\left(\begin{array}{cccccc}
\lambda+2\mu&\lambda&\lambda&0&0&0\\
\lambda&\lambda+2\mu&\lambda&0&0&0\\
\lambda&\lambda&\lambda+2\mu&0&0&0\\
0&0&0&\mu&0&0\\
0&0&0&0&\mu&0\\
0&0&0&0&0&\mu\end{array}\right),
$$
i.e.,
$$\lambda>0\,, \quad
\mu\left\{\begin{array}{ll}
=0 &\mbox{in }\Omega_f\\
>0 &\mbox{in }\Omega_s\end{array}\right. .
$$
(Note that $[c]$ could be set to any symmetric positive definite $6\times6$ matrix with entries in $L^\infty(\Omega)$ in the  elastic region $\Omega_s$, thus allowing for anisotropic elasticity.)
The tensor $[b](x)$ is assumed to be symmetric nonnegative definite and to have the same structure as $[c](x)$ in the fluid region, i.e.,
$$[b]=\left(\begin{array}{cccccc}
\hat{b}&\hat{b}&\hat{b}&0&0&0\\
\hat{b}&\hat{b}&\hat{b}&0&0&0\\
\hat{b}&\hat{b}&\hat{b}&0&0&0\\
0&0&0&0&0&0\\
0&0&0&0&0&0\\
0&0&0&0&0&0\end{array}\right) \mbox{ in }\Omega_f \mbox{ with } \hat{b}(x)\geq\ul{b}>0 \mbox{ in }\Omega_{nl}
\,.
$$
From the Westervelt equation on the subdomain of nonlinearity
$$\Omega_{nl}=\{x\in\Omega \colon \tilde{k}(x)\not=0\}\subseteq\Omega_f$$
we have that
the vector potential $\mathbf{A}$ may be set to zero in the acoustic region ($\psi$ equals the acoustic velocity potential there) and therewith
$$[c]\cB\bu=\lambda\Delta\psi \mathbf{e}\mbox{ in }\Omega_f, \mbox{ where }
\mathbf{e}=(1,1,1,0,0,0)^T,$$
so that on $\Omega_f$ the PDE in \eqref{Welastic_viscosity} becomes
$$
\varrho \nabla\psi_{tt}-\cB^T\frac{\lambda}{1-2\tilde{k}\psi_t}\Delta\psi \mathbf{e}
-\cB^T\Bigl(((1-\delta) +\delta|\Delta\psi_t|^{\qq-1})\hat{b}\Delta\psi_t\mathbf{e}\Bigr)
=0.
$$
Multiplying with an arbitrary vector valued test function $\bv=\nabla w+\nabla\times W$ compactly supported in
$\Omega_f$, integrating by parts on $\Omega$, using the fact that
$$\mathbf{e}^T\cB \bv =\Delta w\,, \quad \text{div}\,\bv=\Delta w$$
and assuming that $\varrho$ is constant on $\Omega_f$ such that
$\varrho \nabla\psi_{tt}= \nabla(\varrho\psi_{tt})$,
we arrive at
$$\int_{\Omega_f} \Delta w
\Bigl(\varrho \psi_{tt}-\frac{\lambda}{1-2\tilde{k}\psi_t}\Delta\psi
-((1-\delta) +\delta|\Delta\psi_t|^{\qq-1})\hat{b}\Delta\psi_t\Bigr)\, dx
=0
$$
for any smooth compactly supported $w$. With
$c^2=\frac{\lambda}{\varrho}$, $b=\frac{\hat{b}}{\varrho}$ we get \eqref{Wpot_viscosity}
which is (up to the damping term) equivalent to \eqref{Westervelt_psi}.
\medskip

In order to be able to make use of the embeddings
\begin{eqnarray}
H_0^1(\Omega)\to L_4(\Omega) \mbox{ with norm }C_{H_0^1,L_4}^\Omega\,,
\label{H01L4}\\
W_0^{1,q+1}(\Omega)\to L_\infty(\Omega) \mbox{ with norm }C_{W_0^{1,q+1},L_\infty}^\Omega\,,
\label{W01qp1linfty}
\end{eqnarray}
in this paper, we will impose zero Dirichlet boundary conditions and assume that $\Omega\subset\R^d$ is an open bounded set with Lipschitz boundary, $d\in\{1,2,3\}$.

Actually, \eqref{H01L4} allows to increase the space dimension even to $d=4$. However, this case is not of practical relevance.
The proofs below show that the embedding $H_0^1(\Omega)\to L_r(\Omega)$ with $r=4$
indeed suffices and we do not need to use the maximal possible exponent $r=6$ in $\R^3$.

Throughout the paper we will use Poincar{\'e}'s inequality; on convex domains the inequality reads (see for example page 364 in \cite{Leoni})
\[
\int_{\Omega} |u(x)-u_{\Omega}|^p dx \leq C(d,p) (\mathrm{diam}(\Omega))^p \int_{\Omega} |\nabla u(x)|^p dx.
\]
where $u_{\Omega}$ is the average of $u$ over the domain $\Omega$ and $\mathrm{diam}(\Omega)$ denotes the diameter of the domain $\Omega$. The Poincar{\'e} constant plays a role in the existence time for the solutions, however, our focus is on establishing local well-posedness of solutions and not on estimating the time of existence.

Several times we will make use of Young's inequality in the form
\begin{equation}\label{abeps}
a b\leq \epsilon a^r+ C(\epsilon,r)b^{\frac{r}{r-1}}
\end{equation}
with
\begin{equation}\label{Cepsr}
C(\epsilon,r)
=(r-1)r^{\frac{r}{r-1}} \epsilon^{-1/(1-r)} \,.
\end{equation}

\medskip
The remainder of this paper consists of five sections, each of them dealing with one of the equations \eqref{Wpress_pLaplace},
\eqref{Wpress_viscosity}, \eqref{Wpot_viscosity}, \eqref{Wacoustic_viscosity} and \eqref{Welastic_viscosity}, respectively.
We prove local in time existence (for sufficiently small times and initial data) for all of them. Global existence and exponential
decay can only be established for the $p$-Laplace damping case \eqref{Wpress_pLaplace}, which lacks uniqueness, though.

\section{The Westervelt equation in acoustic pressure formulation with nonlinear strong damping \eqref{Wpress_viscosity}}
\label{secWpress_viscosity}
First we consider the initial boundary value problem
\begin{equation}\label{W2lin}
\begin{cases}
(1+\alpha) u_{tt}-c^2\Delta u-b\,\text{div}\Bigl(
((1-\delta) +\delta|\nabla u_t|^{q-1})\nabla u_t\Bigr)+fu_t=g,
\\
(u,u_t)|_{t=0}=(u_0, u_1),
\\
u|_{\partial \Omega} =0
.\\
\end{cases}
\end{equation}

\begin{proposition} \label{prop:W2lin}
\begin{enumerate}
\item[(i)]
Let $T>0$, $c^2,b,\delta,1-\delta>0$, $q\geq1$ and assume that
\begin{itemize}
\item
$\alpha\in C(0,T;L_\infty(\Omega))$, $\alpha_t\in L_\infty(0,T;L_2(\Omega))$,
$-1<-\underline{\alpha}\leq \alpha(t,x)\leq \overline{\alpha}$,
\item
$f\in L_\infty(0,T;L_{2}(\Omega))$,
\item
$g\in (L_{q+1}(0,T;W_0^{1,q+1}(\Omega)))^*$,
\item
$u_0\in H_0^1(\Omega)$, $u_1\in L_2(\Omega)$.
\end{itemize}
with
$$\|f-\frac12\alpha_t\|_{L_\infty(0,T;L_{2}(\Omega))}\leq\bar{b}
<\frac{b(1-\delta)}{(C_{H_0^1,L_4}^\Omega)^2}\,. $$

Then \eqref{W2lin} has a weak solution
\begin{align*}
u\in \tilde{X} :=~&C^1(0,T;L_2(\Omega))\cap C(0,T;H_0^1(\Omega))\\
&\cap W^{1,q+1}(0,T;W_0^{1,q+1}(\Omega)).
\end{align*}
Moreover, this solution is unique in $\tilde{X}$ and it satisfies the energy estimate
\begin{eqnarray}
\lefteqn{\frac12\left[\int_\Omega (1+\alpha)(u_t)^2\, dx
+ c^2 |\nabla u|_{L_2(\Omega)}^2 \right]_0^t}
\nonumber\\
&&+ \int_0^t \Bigl( \hat{b}|\nabla u_t|_{L_2(\Omega)}^2
 +\tfrac{b\delta}{2}
|\nabla u_t|_{L_{q+1}(\Omega)}^{q+1}\Bigr)\, ds
\nonumber\\
& \leq &
C(\tfrac{b\delta}{2},q+1) \|g\|_{(L_{q+1}(0,T;W_0^{1,q+1}(\Omega)))^*}^{\frac{q+1}{q}}\,.
\label{enest}
\end{eqnarray}
with
\[
\hat{b}=b(1-\delta)- (C_{H_0^1,L_4}^\Omega)^2 \bar{b}
\]
and $C(\epsilon,r)$ as in \eqref{Cepsr}.

\item[(ii)]
If in addition to (i) $\alpha$ is independent of time, we have the estimate
\begin{equation}
\label{utt}
\begin{aligned}
&\|(1+\alpha)u_{tt}\|_{(L_{q+1}(0,T;W_0^{1,q+1}(\Omega)))^*}\\
&\leq\bar{C} \left\{ \|u\|_{\tilde{X}}
+ \|g\|_{(L_{q+1}(0,T;W_0^{1,q+1}(\Omega)))^*}
+ \|u\|_{\tilde{X}} \|f\|_{L_\infty(0,T;L_2(\Omega))} \right\}.
\end{aligned}
\end{equation}
for some constant $\bar{C}>0$.

\item[(iii)]
If in addition to (i)
\begin{itemize}
\item $f\in L_\infty(0,T;L_4(\Omega))$
with $\|f\|_{L_\infty(0,T;L_4(\Omega))}\leq\bar{\bar{b}}$,
\item $g\in L_2(0,T;L_2(\Omega))$,
\item $u_1\in H_0^1(\Omega)\cap W_0^{1,q+1}(\Omega)$,
\end{itemize} then
\begin{align*}
u\in X:=~& H^2(0,T;L_2(\Omega))\cap C^1(0,T;W_0^{1,q+1}(\Omega)) \,,
\end{align*}
and there exist constants $c_1,C_1$ depending only on $\underline{\alpha}$, $b$, $c$, $\delta$, $C_{H_0^1,L_4}^\Omega$, $\bar{\bar{b}}$,
 such that
\begin{eqnarray}
E_1[u](t)
+c_1\int_0^t \Bigl(
|u_{tt}|_{L_2(\Omega)}^2
+|\nabla u_t|_{L_2(\Omega)}^2
+|\nabla u_t|_{L_{q+1}(\Omega)}^{q+1}
\Bigr)\, ds
&&\nonumber\\
 \leq  C_1 (E_1[u](0)+\|g\|_{L_2(0,T;L_2(\Omega))}
+\|f\|_{L_\infty(0,T;L_4(\Omega))}^2
)
&&
\label{enest1}
\end{eqnarray}
for
\begin{equation}\label{defE1}
E_1[u](t)=\left[|u_t|_{L_2(\Omega)}^2
+ |\nabla u|_{L_2(\Omega)}^2
+ |\nabla u_t|_{L_2(\Omega)}^2
+ |\nabla u_t|_{L_{q+1}(\Omega)}^{q+1}
\right](t).
\end{equation}
\end{enumerate}
\end{proposition}

\begin{remark}
Some simplified versions of Proposition \ref{prop:W2lin} appear in the literature; usually,
$\alpha$ is taken to be zero (i.e. constant coefficient for $u_{tt}$), $\delta= 1$.
This version appears in \cite{Biazutti, Clements} where the damping appears under a regular ($p=2$) Laplacian,
and the authors allow a more general divergence operator, instead of $\Delta u$,
to prove global existence of weak solutions.

Nonlinear $p$-Laplace damping was in considered in \cite{RW}, where a nonlinear source term appears too.
However, other features of the equation (such as variable  $\alpha$) were not present.
For the sake of self-completeness we include below the proof of Proposition \ref{prop:W2lin}.
\end{remark}

\begin{proof}
The weak form of \eqref{W2lin} reads as
\begin{equation}
\label{W2linweak}
\begin{aligned}
\int_\Omega \Bigl\{(1+\alpha)u_{tt} w + c^2 \nabla u \nabla w
+ b\Bigl((1-\delta) +\delta |\nabla u_t|^{q-1}
\Bigr)\nabla u_t\nabla w \Bigr\} \, dx&\\
=\int_\Omega (g -fu_t) w\, dx, \quad \forall w\in W_0^{1,q+1}(\Omega),&
\end{aligned}
\end{equation}
with initial conditions $(u_{0},u_{1})$.\\

{\it Step 1: Smooth approximation of $\alpha$, $f$, and $g$}:
We consider sequences \\
$(\alpha_k)_{k\in\N}$, $(f_k)_{k\in\N}$ and $(g_k)_{k\in\N}$ such that
\begin{itemize}
\item
$(\alpha_k)_{k\in\N}\subseteq C^\infty([0,T]\times\overline{\Omega})\cap W^{1,\infty}(0,T;L_2(\Omega))$,\\
$\alpha_k\to \alpha$ in $C(0,T;L_\infty(\Omega))\cap W^{1,\infty}(0,T;L_2(\Omega))$,
$-1<-\underline{\alpha}\leq \alpha_k(t,x)\leq \overline{\alpha}$,
\item
$(f_k)_{k\in\N}\subseteq C^\infty((0,T)\times\Omega)$,
$f_k\to f$ in $L_\infty(0,T;L_{2}(\Omega))$,
\item
$(g_k)_{k\in\N}\subseteq C^\infty((0,T)\times\Omega)$,
$g_k\to g$ in
$(L_{q+1}(0,T;W_0^{1,q+1}(\Omega)))^*$,
\item
$\|f_k-\frac12\alpha_{k,t}\|_{L_\infty(0,T;L_{2}(\Omega))}\leq\bar{b}$,
\end{itemize}
and, for fixed $k\in\N$, prove that there exists a solution $u^{(k)}$ of \eqref{W2linweak} with $\alpha$, $f$ and $g$ replaced by $\alpha_k$, $f_k$ and $g_k$, respectively, i.e.,
\begin{equation}
\label{W2linweak_k}
\begin{aligned}
\int_\Omega \Bigl\{(1+\alpha_k)u^{(k)}_{tt} w + c^2 \nabla u^{(k)} \nabla w
+ b\Bigl((1-\delta) +\delta |\nabla u^{(k)}_t|^{q-1}
\Bigr)\nabla u^{(k)}_t\nabla w \Bigr\} \, dx&\\
=\int_\Omega (g_k -f_ku^{(k)}_t) w\, dx\quad \forall w\in W_0^{1,q+1}(\Omega),&
\end{aligned}
\end{equation}
with initial conditions $(u_{0},u_{1})$.
Later we will consider limits as $k\to\infty$ to prove well-posedness of \eqref{W2linweak}.
The existence proof follows the line of the standard approach for linear parabolic or second order hyperbolic PDEs as it can be found, e.g., in \cite{Evans}. The proof is divided into three subparts: (a) Galerkin approximation, (b) energy estimates and (c) weak limit.
\\[1ex]
{\it Step 1 (a): Galerkin approximation.} We will first show existence and uniqueness of solutions for a finite-dimensional approximation of \eqref{W2linweak_k}.\\
Assume $w_m=w_m(x)$, $m\in\N$ are smooth functions such that
\begin{align*}
&\{w_m\}_{m\in\N}\text{ is an orthonormal basis of }L_2^{\tilde{\alpha}_k}(\Omega),\\
&\{w_m\}_{m\in\N}\text{ is a basis of } W_0^{1,q+1}(\Omega),
\end{align*}
where $L_2^{\tilde{\alpha}_k}$ is the weighted $L_2$-space based on the inner product
\begin{equation*}
\langle f,g\rangle_{L_2^{\tilde{\alpha}_k}(\Omega)} := \int_\Omega (1+\tilde{\alpha}_k) fg \, dx\,,
\end{equation*}
with
\[
\tilde{\alpha}_k=\frac{1}{T}\int_0^T\alpha(t)\, dt\,.
\]
Moreover, let $V_n$ be the finite dimensional subspace of $L_2^{\alpha_k} (\Omega) \cap W_0^{1,q+1}(\Omega)$ spanned by $\{w_m\}_{m=1}^n$.
Thus $(V_n)_{n\in\N}$ is a nested sequence ($V_k \subseteq V_l$ for $k \leq l$) of finite dimensional subspaces
$V_n\subseteq L_2^{\alpha_k} (\Omega) \cap W_0^{1,q+1}(\Omega)$ such that $\bigcup_{n\in\N} V_n=W_0^{1,q+1}(\Omega)$.\\
Furthermore, suppose $(u_{0,n})_{n\in\N}$, $(u_{1,n})_{n\in\N}$ are sequences satisfying
\begin{itemize}
\item
$u_{0,n}\in  V_n$,
$u_{0,n}\to u_0$ in $H_0^1(\Omega)$,
\item
$u_{1,n}\in  V_n$,
$u_{1,n}\to u_1$ in $L_2(\Omega)$.
\end{itemize}
Based on this, we consider a sequence of discretized versions of \eqref{W2linweak_k},
\begin{equation}
\label{W2linweakdis}
\begin{aligned}
&\int_\Omega \Bigl\{(1+\alpha_k)u_{n,tt}^{(k)} w_n + c^2 \nabla u_n^{(k)} \nabla w_n \\
&\qquad + b\Bigl((1-\delta) +\delta|\nabla u_{n,t}^{(k)}|^{q-1}\Bigr)\nabla u_{n,t}^{(k)} \nabla w_n \Bigr\} \, dx\\
&=\int_\Omega (g_k -f_k u_{n,t}^{(k)})w_n \, dx\quad \forall w_n\in V_n,
\end{aligned}
\end{equation}
with $u_n^{(k)}(t)\in V_n$ and initial conditions $(u_{0,n},u_{1,n})$.\\
For each $n\in\N$ the equality in (\ref{W2linweakdis}) together with initial conditions $(u_{0,n},u_{1,n})$ gives an initial value problem for a second order
system of ordinary differential equations which has smooth (with respect to time) coefficients and right-hand side.\\
The standard existence theory for ordinary differential equations (cf. \cite{Teschl}) provides us with a unique solution
$u_n^{(k)}\in C^\infty(0,\tilde{T};V_n)$ of the finite-dimensional approximation \eqref{W2linweakdis} of \eqref{W2linweak} for some $\tilde{T}\leq T$ sufficiently small.
By the uniform energy and norm estimates below,
we obtain $\tilde{T}=T$, i.e. there is a unique solution $u_n^{(k)}\in C^1(0,T;V_n)$ of \eqref{W2linweakdis}.

{\it Step 1 (b): Energy estimate.} Testing \eqref{W2linweakdis} with $w_n=u_{n,t}^{(k)}(t)$, integrating with respect to time and using the identity
$$\frac{d}{dt}\alpha_k \left(u_{n,t}^{(k)}\right)^2 = 2\alpha_k u_{n,t}^{(k)} u_{n,tt}^{(k)}+\alpha_{k,t} \left(u_{n,t}^{(k)}\right)^2$$
as well as Young's inequality in the form \eqref{abeps} with \eqref{Cepsr},
we obtain
\begin{align*}
&\frac12\left[\int_\Omega (1+\alpha_k)\left(u_{n,t}^{(k)}\right)^2\, dx
+ c^2 |\nabla u_n^{(k)}|_{L_2(\Omega)}^2 \right]_0^t\\
&\quad+ b \int_0^t \int_\Omega \Bigl((1-\delta) +\delta|\nabla u_{n,t}^{(k)}|^{q-1}\Bigr)|\nabla u_{n,t}^{(k)}|^2 \, dx \, ds\\
&=
- \int_0^t  \int_\Omega (f_k-\frac12\alpha_{k,t}) \left(u_{n,t}^{(k)}\right)^2 \, dx \, ds
+ \int_0^t  \int_\Omega g_k u_{n,t}^{(k)} \, dx \, ds\\
&\leq
\bar{b}\int_0^t | u_{n,t}^{(k)}|_{L_4(\Omega)}^2 \, ds
+\int_0^t \Bigl(\tfrac{b\delta}{2} |u_{n,t}^{(k)}|_{W_0^{1,q+1}(\Omega)}^{q+1} \, dx
+C(\tfrac{b\delta}{2},q+1) |g_k|_{(W_0^{1,q+1}(\Omega))^*}^{\frac{q+1}{q}}\Bigr)\, ds\,,
\end{align*}
hence
\begin{equation}\label{enest_kn}
\begin{aligned}
 &\frac12\left[\int_\Omega (1+\alpha_k)\left(u_{n,t}^{(k)}\right)^2\, dx
+ c^2 |\nabla u_{n}^{(k)}|_{L_2(\Omega)}^2 \right]_0^t\\
&+ \int_0^t \Bigl( \hat{b}|\nabla u_{n,t}^{(k)}|_{L_2(\Omega)}^2
+b\delta|\nabla u_{n,t}^{(k)}|_{L_{q+1}(\Omega)}^{q+1}\Bigr)\, ds\\
& \leq
C(\tfrac{b\delta}{2},q+1) \|g_k\|_{(L_{q+1}(0,T;W_0^{1,q+1}(\Omega)))^*}^{\frac{q+1}{q}}\, ,
\end{aligned}
\end{equation}
which corresponds to the energy estimate \eqref{enest} upon replacement of $\alpha,f,g$ by $\alpha_k,f_k,g_k$.
As, by assumption,
$g_k \in (L_{q+1}(0,T;W_0^{1,q+1}(\Omega)))^*$,
we have that $\big(u_n^{(k)}\big)_{n\in\N}$ is a bounded sequence in the Banach space
\begin{equation}\label{space:un}
\begin{aligned}
\tilde{X}:=~&C^1(0,T;L_2(\Omega)) \cap C(0,T;H_0^1(\Omega)) \cap W^{1,q+1}(0,T; W_0^{1,q+1}(\Omega)).
\end{aligned}
\end{equation}
Hence
\begin{equation}\label{Xhat}
\begin{aligned}
\big(u_n^{(k)}\big)_{n\in\N}\text{ is bounded in }W^{1,q+1}(0,T; W_0^{1,q+1}(\Omega)),
\end{aligned}
\end{equation}
which is a reflexive Banach space.
Furthermore, we obtain that
\begin{equation}\label{Lqdual}
|\nabla u_{n,t}^{(k)}|^{q-1} \nabla u_{n,t}^{(k)} \text{ is uniformly bounded in } L_{\frac{q+1}{q}}(0,T;L_{\frac{q+1}{q}}(\Omega)).
\end{equation}
{\it Step 1 (b) --- part (ii): Estimate of $u^{(k)}_{n,tt}$}.
To show (ii) in case $\alpha_{k,t}=0$ and therewith $\tilde{\alpha}_k=\alpha_k$, we now prove an analog of \eqref{utt} with $\alpha_k,f_k,g_k$ in place of $\alpha,f,g$. For this purpose, let $v\in W_0^{1,q+1}(\Omega)$. We decompose $v$,
\begin{equation*}
v= v_n+z_n \text{ with } v_n \in V_n \text{ and } z_n \in V_n^{\bot_{L_2^{\alpha_k}}}.
\end{equation*}
Here
\begin{equation*}
|v_n|_{W_0^{1,q+1}(\Omega)}\leq |v|_{W_0^{1,q+1}(\Omega)}
\end{equation*}
holds due to the fact that with a linearly independent set of dual vectors $\{w_m^*\}_{m\in\N} \subseteq (W_0^{1,q+1}(\Omega))^*$ of $V_n^*$, i.e., such that $\langle w_i^*,w_j\rangle_{W_0^{1,q+1}(\Omega)^*,W_0^{1,q+1}(\Omega)}=\delta_{i,j}$ (which exists by the Hahn-Banach Theorem) we have
\begin{align*}
|v_n|_{W_0^{1,q+1}(\Omega)}
=&\sup_{v^*\in V_n^*\,, \ |v^*|_{W_0^{1,q+1}(\Omega)^*}=1}
\langle v^*,v_n\rangle_{W_0^{1,q+1}(\Omega)^*,W_0^{1,q+1}(\Omega)}\\
=&\sup_{v^*\in \text{span}(w_1^*,\ldots,w_n^*)\,, \ |v^*|_{W_0^{1,q+1}(\Omega)^*}=1}
\langle v^*,v_n\rangle_{W_0^{1,q+1}(\Omega)^*,W_0^{1,q+1}(\Omega)}\\
=&\sup_{v^*\in \text{span}(w_1^*,\ldots,w_n^*)\,, \ |v^*|_{W_0^{1,q+1}(\Omega)^*}=1}
\langle v^*,v\rangle_{W_0^{1,q+1}(\Omega)^*,W_0^{1,q+1}(\Omega)}\\
\leq&\sup_{v^*\in W_0^{1,q+1}(\Omega)^*\,, \ |v^*|_{W_0^{1,q+1}(\Omega)^*}=1}
\langle v^*,v\rangle_{W_0^{1,q+1}(\Omega)^*,W_0^{1,q+1}(\Omega)}\\
=& |v|_{W_0^{1,q+1}(\Omega)}\,.
\end{align*}
 By using orthogonality we obtain the first equality below; by \eqref{W2linweakdis} we obtain the second equality); finally, by using $|.|_{L^2(\Omega)} \leq {C_{L_{q+1},L_2}^\Omega} |.|_{L^{q+1}(\Omega)}$ we have
\begin{align*}
&\int_\Omega (1+\alpha_k)u_{n,tt}^{(k)} v \, dx = \int_\Omega (1+\alpha_k)u_{n,tt}^{(k)} v_n \, dx \\
&=-c^2 \int_\Omega \nabla u_n^{(k)} \nabla v_n \, dx - b(1-\delta) \int_\Omega \nabla u_{n,t}^{(k)} \nabla v_n \, dx\\
&\quad- b\delta \int_\Omega |\nabla u_{n,t}^{(k)}|^{q-1} \nabla u_{n,t}^{(k)} \nabla v_n \, dx + \int_\Omega (g_k - f_k u_{n,t}^{(k)}) v_n \, dx\\
&\leq  \left\{c^2 |\nabla u_n^{(k)}|_{L_2(\Omega)}  + b(1-\delta) |\nabla u_{n,t}^{(k)}| \right\} {C_{L_{q+1},L_2}^\Omega} | \nabla v_n |_{L_{q+1}(\Omega)} \\
&\quad + \left\{b\delta |\nabla u_{n,t}^{(k)} |_{L_{q+1}(\Omega)}^q  + |g_k - f_k u_{n,t}^{(k)}|_{(W_0^{1,q+1}(\Omega))^*}\right\} |\nabla v_n|_{L_{q+1}(\Omega)}\\
&= \left\{c^2 {C_{L_{q+1},L_2}^\Omega} |u_n^{(k)}|_{H^1(\Omega)} + b(1-\delta) {C_{L_{q+1},L_2}^\Omega} |u_{n,t}^{(k)} |_{H^1(\Omega)} \right. \\
&\quad + \left. b\delta |u_{n,t}^{(k)} |_{W_0^{1,q+1}(\Omega)}^q + |g_k - f_k u_{n,t}^{(k)}|_{(W_0^{1,q+1}(\Omega))^*}\right\} \underbrace{|v_n|_{W_0^{1,q+1}(\Omega)}}_{\leq  |v|_{W_0^{1,q+1}(\Omega)}}.
\end{align*}

Multiplying the resulting inequality with a test function $\phi \in C^{\infty}(0,T)$ and integrating with respect to time yields
\begin{align*}
&\int_0^t \int_\Omega (1+\alpha_k)u_{n,tt}^{(k)} v \, dx\, \phi \, ds\\
& \leq  \int_0^t \left\{c^2 {C_{L_{q+1},L_2}^\Omega} |u_n^{(k)}|_{H^1(\Omega)} + b(1-\delta) {C_{L_{q+1},L_2}^\Omega} |u_{n,t}^{(k)} |_{H^1(\Omega)} \right. \\
&\quad + \left. b\delta |u_{n,t}^{(k)} |_{W_0^{1,q+1}(\Omega)}^q + |g_k - f_k u_{n,t}^{(k)}|_{(W_0^{1,q+1}(\Omega))^*}\right\}
|v|_{W_0^{1,q+1}(\Omega)} |\phi| \, ds
\end{align*}
\begin{align*}
&\leq \|v \phi\|_{L_{q+1}(0,T;W_0^{1,q+1}(\Omega))} \left\{ c^2 {C_{L_{q+1},L_2}^\Omega} \|u_n^{(k)}\|_{L_{\frac{q+1}{q}}(0,T;H^1(\Omega))}  \right.\\
&\quad+b(1-\delta) {C_{L_{q+1},L_2}^\Omega} \|u_n^{(k)}\|_{L_{\frac{q+1}{q}}(0,T;H^1(\Omega))} \\
&\quad \left. +b\delta \|u_n^{(k)} \|_{L_{\frac{q+1}{q}}(0,T;W_0^{1,q+1}(\Omega))}+\|g_k-f_k u_{n,t}^{(k)}\|_{(L_{q+1}(0,T;W_0^{1,q+1}(\Omega)))^*)} \right\},
\end{align*}
which shows that
\begin{equation}
\label{unttk}
\begin{aligned}
&\|(1+\alpha_k)u_{n,tt}^{(k)}\|_{(L_{q+1}(0,T; W_0^{1,q+1}(\Omega)))^*}\\
&\leq\bar{C} \left\{ \|u_n^{(k)}\|_{\tilde{X}}
+ \|g_k\|_{(L_{q+1}(0,T; (W_0^{1,q+1}(\Omega)))^*}
+ \|u_n^{(k)}\|_{\tilde{X}} \|f_k\|_{L_\infty(0,T;L_2(\Omega))} \right\}
\end{aligned}
\end{equation}
for some constant $\bar{C}>0$. Hence
$$(1+\alpha_k)u_{n,tt}^{(k)} \in (L_{q+1}(0,T; (W_0^{1,q+1}(\Omega)))^*$$
with a uniform bound with respect to $n$.
\\[1ex]
{\it Step 1 (c): Weak limit.}
As a consequence of \eqref{Xhat}, \eqref{Lqdual},
there exists a weakly convergent subsequence, which for simplicity we denote by $u_n^{(k)}$ again, and a sequence $u^{(k)} \in W^{1,q+1}(0,T; W_0^{1,q+1}(\Omega))$ such that
\begin{align}\label{weakconv1}
&u_n^{(k)}\rightharpoonup u^{(k)} \mbox{ in }W^{1,q+1}(0,T; W_0^{1,q+1}(\Omega)), \\
& |\nabla u_{n,t}^{(k)}|^{q-1} \nabla u_{n,t}^{(k)}
\rightharpoonup |\nabla u_t^{(k)}|^{q-1} \nabla u_t^{(k)}
\mbox{ in }L_{\frac{q+1}{q}} (0,T;L_{\frac{q+1}{q}}(\Omega)).\label{weakconv3}
\end{align}
This weak limit satisfies the estimates \eqref{enest_kn} and \eqref{unttk} with $u_n^{(k)}$ replaced by $u^{(k)}$.\\
For fixed $k,m\in\N$ and $\phi_m\in C^\infty(0,T,V_m)\subset L_{q+1}(0,T;W_0^{1,q+1}(\Omega))$ with $\phi_m(T)=0$ we have for any $n\geq m$ by $V_m\subseteq V_n$ that
\begin{equation}
\label{convntoinfty_W2}
\begin{aligned}
&\int_0^T \int_\Omega \Bigl\{ (1+\alpha_k)u_{tt}^{(k)} \phi_m
+ c^2 \nabla u^{(k)} \nabla \phi_m\\
&\quad+ b\Bigl((1-\delta) +\delta|\nabla u_t^{(k)}|^{q-1}\Bigr)\nabla u_t^{(k)})\nabla \phi_m
 + f_k u_t^{(k)} \phi_m - g_k \phi_m\Bigr\} \, dx\, ds\\
&=-\int_0^T\int_\Omega [u_{t}^{(k)}-u_{n,t}^{(k)}] \Bigl((1+\alpha_k)  \phi_m\Bigr)_t \, dx \, ds\\
&\quad-\int_\Omega [u_1-u_{1,n}] (1+\alpha_k(0))  \phi_m(0) \, dx \, ds + \int_0^T\int_\Omega [u_t^{(k)} -u_{n,t}^{(k)}]f_k \phi_m \, dx\, ds\\
&\quad+  \int_0^T\int_\Omega \left(c^2[\nabla u^{(k)}-\nabla u_n^{(k)}] + b (1-\delta) [\nabla u_t^{(k)}-\nabla u_{n,t}^{(k)}]\right)\nabla \phi_m \, dx\, ds \\
&\quad + b \int_0^T\int_\Omega \delta[|\nabla u_t^{(k)}|^{q-1}\nabla u_t^{(k)}-|\nabla u_{n,t}^{(k)}|^{q-1}\nabla u_{n,t}^{(k)}]\nabla \phi_m\Bigr)
\Bigr\} \, dx\, ds \to 0  \\
\end{aligned}
\end{equation}
$\mbox{ as }n\to \infty$, due to \eqref{weakconv1}
and \eqref{weakconv3}.
Since $\bigcup_{m\in\N} V_m=W_0^{1,q+1}(\Omega)$, the relation  \eqref{convntoinfty_W2} proves that $u^{(k)}$ solves \eqref{W2linweak_k}. Moreover $u^{(k)}$ satisfies the energy estimates \eqref{enest} with $\alpha$, $f$, $g$ and $u$ replaced by $\alpha_k$, $f_k$, $g_k$ and $u^{(k)}$, respectively. In case (ii) $\alpha_{k,t}=0$, so additionally, we have that  \eqref{utt}, with $\alpha$, $f$, $g$ and $u$ replaced by $\alpha_k$, $f_k$, $g_k$ and $u^{(k)}$, respectively, holds.
\\[1ex]

{\it Step 2: $k\to \infty$.} For all $w\in C^\infty(0,T;W_0^{1,q+1}(\Omega))$ with $w(T)=0$ we have
\begin{align*}\nonumber
&\int_0^t \int_\Omega \Bigl\{ (1+\alpha)u_{tt} w + c^2 \nabla u \nabla w\\
&\quad + b\Bigl((1-\delta) +\delta|\nabla u_t|^{q-1}\Bigr)\nabla u_t\nabla w +f u_t w  - gw  \Bigr\} \, dx\, ds
\nonumber\\
&=-\int_0^t \int_\Omega [u_{t}-u_{t}^{(k)}] \Bigl((1+\alpha)w\Bigr)_t \, dx \, ds
- \int_0^t \int_\Omega u_{t}^{(k)} \Bigl([\alpha-\alpha_k]w\Bigr)_t \, dx \, ds\\
&\quad- \int_\Omega u_{1} \Bigl([\alpha(0)-\alpha_k(0)]w(0)\Bigr) \, dx \nonumber\\
&\quad+ \int_0^t \int_\Omega c^2 [\nabla u - \nabla u^{(k)})] \nabla w \, dx \, ds + \int_0^t \int_\Omega  b(1-\delta) [\nabla u_t - \nabla u_t^{(k)}] \nabla w \, dx \, ds \nonumber\\
&\quad + \int_0^t \int_\Omega b\delta [|\nabla u_t|^{q-1}\nabla u_t - |\nabla u_t^{(k)}|^{q-1}\nabla u_t^{(k)}]\nabla w \, dx \, ds \nonumber\\
&\quad + \int_0^t \int_\Omega [u_t - u_t^{(k)}] f w \, dx \, ds + \int_0^t \int_\Omega [f-f_k] u_t^{(k)}  w \, dx\, ds\nonumber\\
&\quad - \int_0^t \int_\Omega [g - g_k] w \, dx \, ds  \to 0 \text{ as } k \to \infty, \nonumber
\end{align*}
since we imposed $\alpha_k \to \alpha$ in $C(0,T;L_\infty(\Omega))\cap W^{1,\infty}(0,T;L_2(\Omega))$, $f_k \to f$ in $L_\infty(0,T;L_2(\Omega))$, $g_k \to g$ in
$(L_{q+1}(0,T;W_0^{1,q+1}(\Omega)))^*$.
Due to the fact that the estimate \eqref{enest} remains valid for $u^{(k)}$, analogously to Step  1(c), we find a convergent subsequence (which
we relabel) and a function
$u\in W^{1,q+1}(0,T; W_0^{1,q+1}(\Omega))$
such that
\begin{align}\label{weakconv4}
&u^{(k)}\rightharpoonup u \mbox{ in }
W^{1,q+1}(0,T; W_0^{1,q+1}(\Omega)),
\\
& |\nabla u_{t}^{(k)}|^{q-1} \nabla u_{t}^{(k)}
\rightharpoonup |\nabla u_t|^{q-1} \nabla u_t
\mbox{ in }L_{\frac{q+1}{q}} (0,T;L_{\frac{q+1}{q}}(\Omega)).\label{weakconv6}
\end{align}
This weak limit is, by construction, a weak solution of \eqref{W2lin} and it satisfies the energy estimate
\eqref{enest}; in case (ii) $\alpha_t=0$, it also satisfies \eqref{utt}.

{\it Step 3: Uniqueness of weak solutions.} Uniqueness of weak solutions follows from the fact that the difference $\hat{u}=u^1-u^2$ between any
two weak solutions $u^1,u^2$ of \eqref{W2lin} satisfies
\begin{equation}\label{W2_uniq}
\begin{cases}
(1+\alpha)\hat{u}_{tt}-c^2\Delta \hat{u}-b(1-\delta)\Delta \hat{u}_t
-b\delta \int_0^1\,\text{div}\Bigl( w^\sigma
\Bigl[|\nabla (u^2+\sigma \hat{u})_t|^2  \nabla \hat{u}_t\\
\quad+(\qq-1)(\nabla (u^2+\sigma \hat{u})_t \nabla \hat{u}_t)
\nabla(u^2+\sigma \hat{u})_t \Bigr]\,  \Bigr)d\sigma
+f\hat{u}_t\ = \ 0
\\
(\hat{u},\hat{u}_t)|_{t=0}=(0,0)
\\
\hat{u}|_{\partial \Omega} =0
,\\
\end{cases}
\end{equation}
in a weak sense, with
\begin{equation}\label{wsig}
w^\sigma(x,t)= |\nabla (u^2+\sigma \hat{u})_t(x,t)|^{\qq-3}\,.
\end{equation}
Above we used the fact that for
\begin{equation*}
F(\lambda)=|\lambda|^{\qq-1}\lambda, \quad F'(\lambda)=|\lambda|^{\qq-1}I + (\qq-1)|\lambda|^{\qq-3}\lambda\lambda^T
\end{equation*}
we have
\begin{equation*}
F(\nabla u^1_t(x,t))-F(\nabla u^2_t(x,t))=
\int_0^1 F'(\nabla (u^2+\sigma\hat{u})_t)(x,t))\, d\sigma\nabla\hat{u}_t(x,t).
\end{equation*}
Multiplication of \eqref{W2_uniq} by $\hat{u}_t$ yields
\[
\begin{aligned}
&\frac12\left[\int_\Omega (1+\alpha)(\hat{u}_t)^2\, dx
+ c^2 |\nabla \hat{u}|_{L_2(\Omega)}^2 \right]_0^t
+ \int_0^t  b(1-\delta) |\nabla \hat{u}_t(s)|^2 \, ds\\
&\hfill+ \int_0^t  \int_\Omega (f-\frac12\alpha_t) (\hat{u}_t)^2 \, dx \, ds
\ \leq \ 0
\end{aligned}
\]
since
\begin{equation}\label{wsig_nonneg}
\begin{aligned}
&b\delta\int_0^1\int_0^t\int_\Omega w^\sigma\Bigl(
|\nabla (u^2+\sigma \hat{u})_t|^2  |\nabla \hat{u}_t|^2\\
&\qquad+(\qq-1)(\nabla (u^2+\sigma \hat{u})_t \cdot\nabla \hat{u}_t)^2
\Bigr)  \, dx\, ds\, d\sigma \geq 0.
\end{aligned}
\end{equation}
Therefore $\hat{u}=0$ almost everywhere and the proof of uniqueness is complete.
\\[1ex]
{\it Step 4: Higher energy estimate.} For the proof of (iii), we need a higher energy estimate  which can be obtained by testing
\eqref{W2linweak} with $w=u_{tt}(t)$ (strictly speaking, we multiply by a smooth approximation of $u_{tt}$ and take weak limits); then integration with respect to time yields
\begin{equation}
\label{enest01A}\begin{aligned}
&\int_0^t\int_\Omega (1+\alpha)(u_{tt})^2\, dx\, ds
+\left[\frac{b(1-\delta)}{2}|\nabla u_{t}|_{L_2(\Omega)}^2
+\frac{b\delta}{q+1}|\nabla u_{t}|_{L_{q+1}(\Omega)}^{q+1}\right]_0^t\\
&=
c^2\int_0^t |\nabla u_{t}|_{L_2(\Omega)}^2\, ds
-c^2\left[\int_\Omega \nabla u \nabla  u_{t}\right]_0^t
-\int_0^t\int_\Omega (f u_t-g)u_{tt}\, dx\, ds\\
& \leq
c^2\int_0^t |\nabla u_{t}|_{L_2(\Omega)}^2\, ds
+\breve{b}\Bigl(|\nabla u_{t}(t)|_{L_2}^2 + |\nabla u_{t}(0)|_{L_2}^2\Bigr)\\
&\quad + \frac{c^4}{4\breve{b}}
\Bigl(|\nabla u(t)|_{L_2}^2 + |\nabla u(0)|_{L_2}^2\Bigr)
+ (\tau+\sigma) \int_0^t|u_{tt}|_{L_2}^2\, ds\\
& \quad+ \frac{\left(C_{H_0^1,L_4}^\Omega\right)^2}{4\tau}
\|f\|_{L_\infty(0,T;L_4(\Omega))}^2
\int_0^t |\nabla u_{t}|_{L_2}^2\, ds
+\frac{1}{4\sigma} \|g\|_{L_2(0,T;L_2(\Omega))},
\end{aligned}
\end{equation}
where we have used integration by parts with respect to time for the $c^2$-term.
Choosing $\breve{b}<\frac{b(1-\delta)}{2}$, then adding \eqref{enest} to  \eqref{enest01A} multiplied by $\lambda$
for any $\tau,\sigma>0$ such that
$$\tau+\sigma<1-\underline{\alpha},~\lambda\in \left(0,\min\left\{\tfrac{\hat{b}}{s}\,,\, 2\tfrac{\breve{b}}{c^2}\right\}\right)
\text{ with } s=2c^2+\tfrac{1}{4\tau}\bar{\bar{b}}(C_{H_0^1,L_4}^\Omega)^2$$
yield the energy estimate \eqref{enest1}.
\end{proof}

Note that by the assumptions $\alpha(t,x)\geq-\underline{\alpha}>-1$, $c^2>0$, \eqref{enest} implies an estimate of the form
\begin{eqnarray}
E_0[u](t)
+c_0\int_0^t \Bigl(
|\nabla u_t|_{L_2(\Omega)}^2
+|\nabla u_t|_{L_{q+1}(\Omega)}^{q+1}
\Bigr)\, ds
&&\nonumber\\
 \leq  C_0 (E_0[u](0)+\|g\|_{(L_{q+1}(0,T;L_{q+1}(\Omega)))^*}) &&
\label{enest0}
\end{eqnarray}
for the usual lower order energy
\begin{equation}\label{defE0}
E_0[u](t)=\left[|u_t|_{L_2(\Omega)}^2 +|\nabla u|_{L_2(\Omega)}^2 \right](t).
\end{equation}

Using Proposition \ref{prop:W2lin} in a fixed point argument, we now show local well-posedness.
\begin{theorem}\label{th:W2locex}
Let $c^2,b,\delta,1-\delta>0$, $k\in\R$, $\qq>d-1$.\\
For any $T>0$ there is a $\kappa_T>0$ such that for all $u_0,u_1\in W_0^{1,\qq+1}(\Omega)$ with
\[\begin{aligned}
&E_1[u](0)+|\nabla u_0|_{L_{\qq+1}}^2 \\
&= |u_1|_{L_2(\Omega)}^2 +|\nabla u_0|_{L_2(\Omega)}^2
+ |\nabla u_1|_{L_2(\Omega)}^2
+ |\nabla u_1|_{L_{q+1}(\Omega)}^{q+1}
+|\nabla u_0|_{L_{\qq+1}}^2\leq \kappa_T^2
\end{aligned}\]
there exists a weak solution $u\in \cW $ of \eqref{Wpress_viscosity} where
\begin{eqnarray}\label{defcW}
\cW =\{v\in X
&:& \|v_{tt}\|_{L_2(0,T;L_2(\Omega))}\leq \bar{m}\nonumber\\
&& \wedge \,  \|\nabla v_t\|_{L_\infty(0,T;L_2(\Omega))}\leq \bar{m}\nonumber\\
&& \wedge\,  \|\nabla v_t\|_{L_{q+1}(0,T;L_{q+1}(\Omega))}\leq \bar{M}\}
\end{eqnarray}
with
\begin{equation}\label{smallnessMbar}
2k C_{W_0^{1,\qq+1},L_\infty}^\Omega
(\kappa_T+ T^{\frac{\qq}{\qq+1}} \bar{M}) <1
\end{equation}
and $\bar{m}$ sufficiently small, and $u$ is unique in $\cW $.
\end{theorem}
\begin{proof}
We define the fixed point operator $\cT :\cW \to X$, $v\mapsto \cT v=u$ where $u$ solves \eqref{W2lin} with
\begin{equation}\label{alphafg_W}
\alpha=2k v\,, \quad f=2kv_t\,, \quad g=0\,,
\end{equation}
which is well-defined by Proposition \ref{prop:W2lin}.
Indeed, by $v\in \cW $, \eqref{smallnessMbar}, and the penultimate line in \eqref{estLinftyW2} we have that
$\alpha\in C(0,T;L_\infty(\Omega))$, $-1<-\underline{\alpha}\leq \alpha(t,x)\leq \overline{\alpha}$, with $\underline{\alpha}=\overline{\alpha}=
2k C_{W_0^{1,\qq+1},L_\infty}^\Omega (\kappa_T+ T^{\frac{\qq}{\qq+1}} \bar{M}) <1$.
Moreover,
\begin{align*}
\|f-\tfrac{1}{2}\alpha_t\|_{L_\infty(0,T;L_{2}(\Omega)}&\leq k C_{PF} \bar{m},\\
\|f\|_{L_\infty(0,T;L_4(\Omega))}^2&\leq 4k^2 (C_{H_0^1,L_4}^\Omega)^2 \bar{m}^2,
\end{align*}
with $C_{PF}$ the constant in the Poincar\'{e}-Friedrichs inequality
$$ \|w\|_{L_2(\Omega)} \leq C_{PF} \|\nabla w\|_{L_2(\Omega)}\,.$$
Hence, we can make use of the higher energy estimate \eqref{enest1} to conclude that for any $\bar{m},\bar{M}>0$ with

\[
2k C_{W_0^{1,\qq+1},L_\infty}^\Omega T^{\frac{\qq}{\qq+1}} \bar{M} <1,
\quad
k C_{PF} \bar{m}<\frac{1-\delta}{(C_{H_0^1,L_4}^\Omega)^2},
\]

we obtain that under the assumption

\[
E_1[u](0)\leq \kappa_T<
\min\Big\{\frac{1}{2k C_{W_0^{1,\qq+1},L_\infty}^\Omega}-T^{\frac{\qq}{\qq+1}} \bar{M}, \frac{\bar{m}}{\sqrt{C_1}}, \sqrt{\frac{c_1}{C_1}}\bar{m}\Big\},
\]

the operator $\cT$ maps into $\cW $.

Contractivity is obtained by considering $v^i\in\cW$, $u^i=\cT v^i\in\cW$, $i=1,2$ and subtracting the equations for $u^1$ and $u^2$, which yields
\begin{equation}\label{W2_diff}
\begin{cases}
(1-2kv^1)\hat{u}_{tt}
-b\delta\int_0^1\,\text{div}\Bigl( w^\sigma
\Bigl[|\nabla (u^2+\sigma \hat{u})_t|^2  \nabla \hat{u}_t\\
\quad+(\qq-1)(\nabla (u^2+\sigma \hat{u})_t \nabla \hat{u}_t)
\nabla(u^2+\sigma \hat{u})_t \Bigr]\,  \Bigr)d\sigma \\
\quad -b(1-\delta)\Delta \hat{u}_t -c^2\Delta \hat{u}
-2k v^1_t \hat{u}_t = \ 2k(\hat{v}u^2_{tt}+\hat{v}_tu^2_t),
\\
(\hat{u},\hat{u}_t)|_{t=0}=(0,0),
\\
\hat{u}|_{\partial \Omega} =0,
\\
\end{cases}
\end{equation}
for $\hat{u}=u^1-u^2$, $\hat{v}=v^1-v^2$, with $w^\sigma(x,t)=|\nabla(u^2+\sigma\hat{u})_t(x,t)|^{q-3}$ as in \eqref{wsig}.
Due to the special form of the nonlinear strong damping term here, we cannot apply Proposition \ref{prop:W2lin} directly,
but we can proceed analogously to its proof. By multiplication of \eqref{W2_diff} with $\hat{u}_t$, integration
with respect to space and time, and the fact that the $b\delta$ term yields a nonnegative contribution
\begin{equation}\label{wsig_nonneg1}
\begin{aligned}
&b\delta\int_0^1\int_0^t\int_\Omega w^\sigma\Bigl(
|\nabla (u^2+\sigma \hat{u})_t|^2  |\nabla \hat{u}_t|^2\\
&\qquad+(\qq-1)(\nabla (u^2+\sigma \hat{u})_t \cdot\nabla \hat{u}_t)^2
\Bigr)  \, dx\, ds\, d\sigma \geq 0.
\end{aligned}
\end{equation}
on the left hand side,
%
%
we obtain
\begin{eqnarray*}
\lefteqn{\frac12\left[\int_\Omega (1-2kv^1)(\hat{u}_t)^2\, dx
+ c^2 |\nabla \hat{u}|_{L_2(\Omega)}^2 \right]_0^t
+ \hat{b} \int_0^t |\nabla \hat{u}_t|_{L_2(\Omega)}^2
\, ds}
\nonumber\\
& \leq & 2k \int_0^t\int_\Omega \Bigl(
v^1_t (\hat{u}_t)^2 +\hat{v}u^2_{tt}\hat{u}_t+\hat{v}_tu^2_t\hat{u}_t
\Bigr)\, dx\, ds\\
& \leq & 2k (C_{H_0^1,L_4}^\Omega)^2\Bigl(
\|v^1_t\|_{L_\infty(0,T;L_2(\Omega))}
\int_0^t |\nabla \hat{u}_t|_{L_2(\Omega)}^2\, ds\\
&&\qquad+\|u^2_{tt}\|_{L_2(0,T;L_2(\Omega))}
\frac12\Bigl[\|\nabla\hat{v}\|_{L_\infty(0,T;L_2(\Omega))}^2+
\int_0^t |\nabla \hat{u}_t|_{L_2(\Omega)}^2\, ds\Bigr]\\
&&\qquad+\|u^2_t\|_{L_\infty(0,T;L_2(\Omega))}
\frac12\Bigl[\|\nabla\hat{v}_t\|_{L_2(0,T;L_2(\Omega))}^2+
\int_0^t |\nabla \hat{u}_t|_{L_2(\Omega)}^2\, ds\Bigr].
\end{eqnarray*}
This by $v^1,v^2,u^1,u^2\in\cW$ yields
\begin{eqnarray*}
\lefteqn{\min\Big\{\tfrac{1-\underline{\alpha}}{2},\tfrac{c^2}{2},
\hat{b}- k (C_{H_0^1,L_4}^\Omega)^2 (3 C_{PF}+1)\bar{m}\Big\}
|||\hat{u}|||}\\
&\leq&  k (C_{H_0^1,L_4}^\Omega)^2 (\sqrt{T}+C_{PF})\bar{m}
|||\hat{v}|||\,,
\end{eqnarray*}
where we have used  $|\nabla \hat{v}|_{L_\infty(0,T;L_2(\Omega))}^2
\leq T |\nabla \hat{v}_t|_{L_2(0,T;L_2(\Omega))}^2$ since $\nabla \hat{v}(0)=0$,
hence contractivity of $\cT$ with respect to the norm
$$|||v|||=\|v_t\|_{L_\infty(0,T;L_2(\Omega))}
+\|\nabla v\|_{L_\infty(0,T;L_2(\Omega))}
+\|\nabla v_t\|_{L_2(0,T;L_2(\Omega))}$$
provided $\bar{m}$ is chosen sufficiently small.
\end{proof}

\begin{remark}\label{rem:W2glob}

To establish a self-mapping property of the fixed point operator used in the local well-posedness proof above, a condition on smallness of the initial data multiplied with the final time is required, cf. \eqref{smallnessMbar}. This results from the appearance of the factor $t^\qq$ in the $L_\infty$ estimate \eqref{estLinftyW2}, that we have used to exclude the possibility of degeneracy. Thus, whenever \eqref{estLinftyW2} is sharp, degeneracy may occur after finite time even for small initial data. We therefore expect a global in time wellposedness result not to hold for \eqref{Wpress_viscosity}.

\end{remark}


\section{The Westervelt equation in acoustic velocity potential formulation with nonlinear strong damping \eqref{Wpot_viscosity}}
\label{secWpot_viscosity}

Like in Section \ref{secWpress_viscosity} we assume
that $\Omega\subseteq \R^d$, $d \in \{1,2,3\}$, is an open bounded set with Lipschitz boundary in order to make use of \eqref{H01L4}, \eqref{W01qp1linfty}.
\\
Again we first consider an equation in which the nonlinearity occurs only through damping, namely
\begin{equation}\label{Wpot_viscosity_lin}
\begin{cases}
\psii_{tt}-\alpha\Delta \psii
-b\,\text{div}\Bigl(((1-\delta) +\delta|\nabla \psii_t|^{\qq-1})\nabla \psii_t\Bigr)
=0
\\
(\psii,\psii_t)|_{t=0}=(\psii_0, \psii_1)
\\
\psii|_{\partial \Omega} =0
.\\
\end{cases}
\end{equation}
\begin{proposition} \label{prop:Wpot_viscosity_lin}
\begin{enumerate}
\item[(i)]
Let $T>0$, $b,\delta,1-\delta>0$ and assume
\begin{itemize}
\item $\alpha\in L_2(0,T;L_\infty(\Omega))$, $\nabla\alpha\in L_2(0,T;L_2(\Omega))$
\item $u_0\in H_0^1(\Omega)$.
\item $\qq\geq1$ if $d=1$, $\qq>d-1$ if $d\in \{2,3\}$
\end{itemize}
and
\[\begin{cases}
b-(\tfrac12+\sqrt{T})\|\alpha\|_{L_2(0,T;L_\infty(\Omega))}\\
\qquad-(\tfrac12 T+({C_{W_0^{1,\qq+1},L_\infty}^\Omega})^2)\|\nabla\alpha\|_{L_2(0,T;L_2(\Omega))}:=b_1>0 &
\mbox{ if }\qq=1,\\
b(1-\delta)-(\tfrac12+\sqrt{T})\|\alpha\|_{L_2(0,T;L_\infty(\Omega))}\\ \hfill
-\tfrac12 T\|\nabla\alpha\|_{L_2(0,T;L_2(\Omega))}:=b_1>0 & \mbox{ if }\qq>1.
\end{cases}\]
Then \eqref{Wpot_viscosity_lin} has a weak solution
\begin{align*}
u\in C^1([0,T];L_2(\Omega))\cap W^{1,\qq+1}(0,T;W_0^{1,\qq+1}(\Omega))
\end{align*}
and any solution satisifies the energy estimate
\begin{equation}\label{enest_Wpot_viscosity_lin0}
\begin{aligned}
&\tfrac12\left[|\psii_{t}|_{L_2(\Omega)}^2\right]_0^t
+\min\{b_1,\tfrac{b\delta}{2}\}\int_0^t \Bigl(|\nabla\psii_{t}|_{L_2(\Omega)}^2
+|\nabla\psii_{t}|_{L_{\qq+1}(\Omega)}^{\qq+1}\Bigr)\, ds
\\
&\leq
\begin{cases}
\Bigl(\|\alpha\|_{L_2(0,T;L_\infty(\Omega))} +\|\nabla\alpha\|_{L_2(0,T;L_2(\Omega))}\Bigr)
\tfrac12 |\nabla u_0|_{L_2(\Omega)} &\mbox{ if }\qq=1,\\
\Bigl(\|\alpha\|_{L_2(0,T;L_\infty(\Omega))} +\|\nabla\alpha\|_{L_2(0,T;L_2(\Omega))}\Bigr)
\tfrac12 |\nabla u_0|_{L_2(\Omega)} \\
\hfill+C(\tfrac{b\delta}{2},\tfrac{\qq+1}{2})\|\nabla\alpha\|_{L_2(0,T;L_2(\Omega))}^{\frac{\qq+1}{\qq-1}}
&\mbox{ if }\qq>1,
\end{cases}
\end{aligned}
\end{equation}
for $C(\tfrac{b\delta}{2},\tfrac{\qq+1}{2})$ according to \eqref{Cepsr}.
%
\item[(ii)]
Let $T>0$, $b,\delta,1-\delta>0$, and assume that
\begin{itemize}
\item
$\alpha(t,x)\geq\ul{\alpha}>0$
\item
$\alpha\in C(0,T;L_\infty(\Omega))$, $\alpha_t\in L^{4/3}(0,T;L^2(\Omega))$, $\nabla\alpha\in L^2(0,T;L^4(\Omega))$
\item
$\|\nabla\alpha\|_{L^2(0,T;L^4(\Omega))}< 1$
\item
$\|\alpha_t\|_{L_2(0,T;L_\infty(\Omega))}+\|\nabla\alpha\|_{L^2(0,T;L^4(\Omega))}<\frac{\ul{\alpha}}{2}$
\item
$u_0\in H_0^1(\Omega)$, $u_1\in W_0^{1,\qq+1}(\Omega)$.
\item
$\qq\geq 3$
\end{itemize}
Then \eqref{Wpot_viscosity_lin} has a weak solution
\begin{align*}
u\in H^2(0,T;L^2(\Omega))\cap C^1(0,T,W_0^{1,\qq+1}(\Omega))
\end{align*}
and any solution satisfies the energy estimate
\begin{equation}\label{enest_Wpot_viscosity_lin2}
\begin{aligned}
&\tfrac12\|\psii_{t}\|_{C(0,t;L_2(\Omega))}^2+\|\sqrt{\alpha}\nabla\psii\|_{C(0,t;L_2(\Omega))}^2
+\|\nabla\psii_{t}\|_{L^2(0,t;L_2(\Omega))}^2 \\
&\quad+\|\nabla\psii_{t}\|_{L_{\qq+1}(0,t;L_{\qq+1}(\Omega))}^{\qq+1}\\
&\quad+\|\psii_{tt}\|_{L^2(0,t;L_2(\Omega))}^2
+\|\nabla\psii_{t}\|_{C(0,t;L_2(\Omega))}^2
+|\nabla\psii_{t}|_{C(0,t;L_{\qq+1}(\Omega))}^{\qq+1}\\
&\leq
\bar{C}\Bigl(
\|\nabla\alpha\|_{L_2(0,T;L_4(\Omega))}^{\frac{\qq+1}{\qq-1}}
+|u_1|_{L_2(\Omega)}^2+ \|\alpha\|_{C(0,T;L^\infty(\Omega))}|\nabla u_0|_{L_2(\Omega)}^2\\
&\quad+(\|\alpha_t\|_{L^{4/3}(0,T;L_2(\Omega))}
+\|\nabla\alpha\|_{L_2(0,T;L_4(\Omega))}) |\nabla u_0|_{L_4(\Omega)}^2\\
&\quad+\|\alpha\|_{C(0,T;L_\infty(\Omega))}
\Bigl(|\nabla u_{0}|_{L_2(\Omega)}^2+|\nabla u_1|_{L^2(\Omega)} |\nabla u_0|_{L_2(\Omega)}\Bigr)\\
&\quad+\|\alpha_t\|_{L^{4/3}(0,T;L_2(\Omega))}^{\frac{\qq+1}{\qq-1}}
+(t^{3/2} \|\nabla\alpha\|_{L_2(0,T;L_4(\Omega))})^{\frac{\qq+1}{\qq-1}}\\
&\quad+|\nabla\psii_{1}|_{L_2(\Omega)}^2
+|\nabla\psii_{1}|_{L_{\qq+1}(\Omega)}^{\qq+1}
\Bigr)
\end{aligned}
\end{equation}
for some constant $\bar{C}>0$.
\end{enumerate}
\end{proposition}
\begin{proof}
Since we deal with an autonomous second order in time PDE, the proof can be done directly via Galerkin discretization, energy estimates and weak limits.
\\
We therefore only provide the crucial energy estimates for (i) and (ii),
which are obtained by multiplying the discretized version of \eqref{Wpot_viscosity_lin} with $\psii_{n,t}$ and $\psii_{n,tt}$, respectively and integrating by parts with respect to space and time.
\\

For (i), from multiplication of the discretized version of \eqref{Wpot_viscosity_lin} with $\psii_{n,t}$ we get
\begin{equation*}
\begin{aligned}
&\tfrac12\left[|\psii_{n,t}|_{L_2(\Omega)}^2\right]_0^t
+\int_0^t \left(b(1-\delta)|\nabla\psii_{n,t}|_{L_2(\Omega)}^2
+b\delta|\nabla\psii_{n,t}|_{L_{\qq+1}(\Omega)}^{\qq+1}\right)\, ds
\\
&=\int_0^t\int_\Omega\Bigl( -\alpha \nabla \psii_{n,t}\nabla \psii_{n} -\psii_{n,t}\nabla\alpha \nabla \psii_{n}
\Bigr) \, dx \, ds\\
&\leq \int_0^t \Bigl(|\alpha(s)|_{L_\infty(\Omega)} |\nabla \psii_{n,t}(s)|_{L_2(\Omega)}
+|\nabla\alpha(s)|_{L_2(\Omega)}|\psii_{n,t}(s)|_{L_\infty(\Omega)}\Bigr)\\
&\qquad\left[|\nabla u_{0,n}|_{L_2(\Omega)}+\sqrt{s\int_0^s|\nabla \psii_{n,t}(\sigma)|_{L_2(\Omega)}^{2}\,d\sigma}\right] \, ds\\
&\leq \left(\|\alpha\|_{L_2(0,T;L_\infty(\Omega))} \sqrt{\int_0^t |\nabla \psii_{n,t}(s)|_{L_2(\Omega)}^2\, ds}
+\|\nabla\alpha\|_{L_2(0,T;L_2(\Omega))}\sqrt{\int_0^t |\psii_{n,t}(s)|_{L_\infty(\Omega)}^2\, ds}\right)\\
& \quad\cdot\left[|\nabla u_{0,n}|_{L_2(\Omega)}+\sqrt{t\int_0^t|\nabla \psii_{n,t}(\sigma)|_{L_2(\Omega)}^{2}\,d\sigma}\right]
\end{aligned}
\end{equation*}
which in case $q=1$ yields
\begin{equation}\label{Wpot_visc_lin_enid0}
\begin{aligned}
&\tfrac12\left[|\psii_{n,t}|_{L_2(\Omega)}^2\right]_0^t
+\int_0^t \Bigl(b(1-\delta)|\nabla\psii_{n,t}|_{L_2(\Omega)}^2
+b\delta|\nabla\psii_{n,t}|_{L_{\qq+1}(\Omega)}^{\qq+1}\Bigr)\, ds
\\
&\leq \|\alpha\|_{L_2(0,T;L_\infty(\Omega))} \Bigl( \sqrt{t}\int_0^t  |\nabla \psii_{n,t}(s)|_{L_2(\Omega)}^2\, ds +\tfrac12 \int_0^t |\nabla \psii_{n,t}(s)|_{L_2(\Omega)}^2\, ds
+\tfrac12 |\nabla u_{0,n}|_{L_2}^2\Bigr)\\
&\quad+\|\nabla\alpha\|_{L_2(0,T;L_2(\Omega))}
\Bigl(({C_{W_0^{1,\qq+1},L_\infty}^\Omega})^2\int_0^t |\nabla\psii_{n,t}(s)|_{L_{\qq+1}\Omega)}^2\, ds\\
&\qquad+\tfrac12 t\int_0^t |\nabla \psii_{n,t}(s)|_{L_2(\Omega)}^2\, ds +\tfrac12 |\nabla u_{0,n}|_{L_2}^2\Bigr).
\end{aligned}
\end{equation}
If $\qq>1$, using \eqref{abeps} and \eqref{Cepsr}, we further estimate
\begin{equation}\label{Wpot_visc_lin_enid0_qgt1}
\begin{aligned}
&\tfrac12\left[|\psii_{n,t}|_{L_2(\Omega)}^2\right]_0^t
+\int_0^t \Bigl(b(1-\delta)|\nabla\psii_{n,t}|_{L_2(\Omega)}^2
+b\delta|\nabla\psii_{n,t}|_{L_{\qq+1}(\Omega)}^{\qq+1}\Bigr)\, ds\\
&\leq \|\alpha\|_{L_2(0,T;L_\infty(\Omega))} \Bigl( \sqrt{t}\int_0^t  |\nabla \psii_{n,t}(s)|_{L_2(\Omega)}^2\, ds \\
&\quad+\tfrac12 \int_0^t |\nabla \psii_{n,t}(s)|_{L_2(\Omega)}^2\, ds
+\tfrac12 |\nabla u_{0,n}|_{L_2}^2\Bigr)+\tfrac{b\delta}{2}\int_0^t |\nabla\psii_{n,t}(s)|_{L^{\qq+1}\Omega)}^{\qq+1}\, ds\\
&\quad+\|\nabla\alpha\|_{L_2(0,T;L_2(\Omega))}
\Bigl(\tfrac12 t\int_0^t |\nabla \psii_{n,t}(s)|_{L_2(\Omega)}^2\, ds +\tfrac12 |\nabla u_{0,n}|_{L_2}^2\Bigr)\\
&
\quad+C(\tfrac{b\delta}{2},\tfrac{\qq+1}{2}) t (({C_{W_0^{1,\qq+1},L_\infty}^\Omega})^2 \|\nabla\alpha\|_{L_2(0,T;L_2(\Omega))})^{\frac{\qq+1}{\qq-1}}\,.
\end{aligned}
\end{equation}
By our assumptions on smallness of $\alpha$, the estimates \eqref{Wpot_visc_lin_enid0} and \eqref{Wpot_visc_lin_enid0_qgt1} yield
\begin{equation}\label{Wpot_visc_lin_enid0_1}
\begin{aligned}
&\tfrac12\left[|\psii_{n,t}|_{L_2(\Omega)}^2\right]_0^t
+\min\{b_1,\tfrac{b\delta}{2}\}\int_0^t \Bigl(|\nabla\psii_{n,t}|_{L_2(\Omega)}^2
+|\nabla\psii_{n,t}|_{L_{\qq+1}(\Omega)}^{\qq+1}\Bigr)\, ds
\\
\leq&
\begin{cases}
\Bigl(\|\alpha\|_{L_2(0,T;L_\infty(\Omega))} +\|\nabla\alpha\|_{L_2(0,T;L_2(\Omega))}\Bigr)\tfrac12 |\nabla u_{0,n}|_{L_2}^2&\mbox{ if }\qq=1,\\
\Bigl(\|\alpha\|_{L_2(0,T;L_\infty(\Omega))} +\|\nabla\alpha\|_{L_2(0,T;L_2(\Omega))}\Bigr)\tfrac12 |\nabla u_{0,n}|_{L_2}^2\\
\hfill+ C(\tfrac{b\delta}{2},\tfrac{\qq+1}{2}) t \|\nabla\alpha\|_{L_2(0,T;L_2(\Omega))}^{\frac{\qq+1}{\qq-1}}&\mbox{ if }\qq>1.
\end{cases}
\end{aligned}
\end{equation}
For (ii),
we  multiply \eqref{Wpot_viscosity_lin} with $\psii_{n,t}$ (but carry out the estimates after multiplication differently) and with $u_{n,tt}$.
By multiplication with $\psii_{n,t}$ we get
\begin{equation}\label{Wpot_visc_lin_enid0_iii}
\begin{aligned}
&\tfrac12\left[|\psii_{n,t}|_{L_2(\Omega)}^2+|\sqrt{\alpha}\nabla\psii_{n}|_{L_2(\Omega)}^2\right]_0^t+\int_0^t \Bigl(b(1-\delta)|\nabla\psii_{n,t}|_{L_2(\Omega)}^2
+b\delta|\nabla\psii_{n,t}|_{L_{\qq+1}(\Omega)}^{\qq+1}\Bigr)\, ds\\
&=\int_0^t\int_\Omega\Bigl( \tfrac12\alpha_t |\nabla \psii_{n}|^2 -\psii_{n,t}\nabla\alpha \nabla \psii_{n}
\Bigr) \, dx \, ds\\
&\leq \tfrac12\|\alpha_t\|_{L_2(0,T;L_\infty(\Omega))} \|\nabla \psii_{n}\|_{C(0,T;L_2(\Omega))}^2\\
&\quad+\|\nabla\alpha\|_{L_2(0,T;L_4(\Omega))}\Bigl(\tfrac12 \|\psii_{n,t}\|_{C(0,T;L_4(\Omega))}^2
+\tfrac12 \|\nabla \psii_{n}\|_{C(0,T;L_2(\Omega))}^2\Bigr)
\end{aligned}
\end{equation}
which due to our smallness assumption on $\alpha$ implies
\begin{equation}\label{Wpot_visc_lin_enid0_iii_1}
\begin{aligned}
&\tfrac12|\psii_{n,t}(t)|_{L_2(\Omega)}^2+c_1|\sqrt{\alpha}\nabla\psii_{n}(t)|_{L_2(\Omega)}^2+\int_0^t \Bigl(b(1-\delta)|\nabla\psii_{n,t}|_{L_2(\Omega)}^2
+b\delta|\nabla\psii_{n,t}|_{L_{\qq+1}(\Omega)}^{\qq+1}\Bigr)\, ds
\\
&\leq \epsilon_0 \|\psii_{n,t}\|_{C(0,T;L_4(\Omega))}^{\qq+1}+
C(\epsilon_0,\tfrac{\qq+1}{2}) (\tfrac12\|\nabla\alpha\|_{L_2(0,T;L_4(\Omega))})^{\frac{\qq+1}{\qq-1}} \\
&\quad+\tfrac12|u_{1,n}|_{L_2(\Omega)}^2+\tfrac12 \|\alpha\|_{C(0,T;L^\infty(\Omega))}|\nabla u_{0,n}|_{L_2(\Omega)}^2
\end{aligned}
\end{equation}
for $c_1=\frac{\ul{\alpha}}{2}-\|\alpha_t\|_{L_2(0,T;L_\infty(\Omega))}-\|\nabla\alpha\|_{L^2(0,T;L^4(\Omega))}$ and some $\epsilon_0 >0$.
\\
Multiplication with $\psii_{n,tt}$ yields
\begin{align*}
&\int_0^t |\psii_{n,tt}(s)|_{L_2(\Omega)}^2\, ds
+\left[\tfrac{b(1-\delta)}{2}|\nabla\psii_{n,t}|_{L_2}^2
+\tfrac{b\delta}{\qq+1}|\nabla\psii_{n,t}|_{L_{\qq+1}(\Omega)}^{\qq+1}
\right]_0^t
\\
&=\int_0^t\int_\Omega\left( -\alpha \nabla \psii_{n,tt}\nabla \psii_{n} -\psii_{n,tt}\nabla\alpha \nabla \psii_n
\right) \, dx \, ds\\
&=\int_0^t\int_\Omega\left( \alpha_t \nabla \psii_{n,t}\nabla \psii_{n}+\alpha|\nabla \psii_{n,t}|^2
-\psii_{n,tt}\nabla\alpha \nabla \psii_n\right)dx\, ds-\left[\int_\Omega\left( \alpha \nabla \psii_{n,t}\nabla \psii_{n} \right)\,dx\right]_0^t\\
& \leq
\int_0^t\left( |\alpha_t(s)|_{L_2(\Omega)} |\nabla \psii_{n,t}(s)|_{L_4(\Omega)}
+|\nabla\alpha(s)|_{L^4(\Omega)}|\psii_{n,tt}(s)|_{L^2(\Omega)} \right) \\
&\quad\left[|\nabla u_{0,n}|_{L_4(\Omega)}+\sqrt[4]{s^3\int_0^s|\nabla \psii_{n,t}(\sigma)|_{L_4(\Omega)}^{4}\,d\sigma}\right] \, ds
\\
&+\|\alpha\|_{C(0,T;L_\infty(\Omega))} \Bigl(|\nabla \psii_{n,t}(t)|_{L^2(\Omega)}
|\nabla \psii_{n}(t)|_{L^2(\Omega)}+|\nabla u_{1,n}|_{L^2(\Omega)} |\nabla u_{0,n}|_{L_2(\Omega)}+\int_0^t|\nabla \psii_{n,t}(s)|_{L^2(\Omega)}^2\, ds\Bigr)\\
&\leq
\|\alpha_t\|_{L^{4/3}(0,T;L_2(\Omega))} \left((1+\tfrac12 t^{\frac{3}{2}})\|\nabla \psii_{n,t}\|_{L^4(0,t;L_4(\Omega))}^2+\tfrac12|\nabla u_{0,n}|_{L_4(\Omega)}^2\right)\\
&+\|\nabla\alpha\|_{L_2(0,T;L_4(\Omega))} \left(\|\psii_{n,tt}\|_{L^2(0,t;L_2(\Omega))}^2
+\tfrac12 |\nabla u_{0,n}|_{L_4(\Omega)}^2 +\frac{t^{\frac{3}{2}} }{2}\|\nabla \psii_{n,t}\|_{L^4(0,T;L_4(\Omega))}^2\right)\\
&\quad+\|\alpha\|_{C(0,T;L_\infty(\Omega))} \Bigl(|\nabla u_{1,n}|_{L^2(\Omega)} |\nabla u_{0,n}|_{L_2(\Omega)}
+\|\nabla \psii_{n,t}\|_{L^2(0,t;L_2(\Omega))}^2\Bigr)\\
&\quad+\tfrac{b(1-\delta)}{4}|\nabla \psii_{n,t}(t)|_{L^2(\Omega)}^2
+\tfrac{1}{b(1-\delta)}\|\alpha\|_{C(0,T;L_\infty(\Omega))}^2|\nabla \psii_{n}(t)|_{L^2(\Omega)}^2\\
&\leq
\|\alpha_t\|_{L^{4/3}(0,T;L_2(\Omega))}\tfrac12|\nabla u_{0,n}|_{L_4(\Omega)}^2\\
&\quad+\|\nabla \psii_{n,t}\|_{L^4(0,T;L_4(\Omega))}^{\qq+1}
+C(1,\tfrac{\qq+1}{2})((1+\tfrac12 t^{\frac{3}{2}})\|\alpha_t\|_{L^{4/3}(0,T;L_2(\Omega))})^{\frac{\qq+1}{\qq-1}}\\
&\quad+\|\nabla\alpha\|_{L_2(0,T;L_4(\Omega))} \Bigl(\|\psii_{n,tt}(s)|_{L^2(0,t;L_2(\Omega))}^2
+\tfrac12 |\nabla u_{0,n}|_{L_4(\Omega)}^2\Bigr)\\
&\quad+\|\nabla \psii_{n,t}\|_{L^4(0,T;L_4(\Omega))}^{\qq+1}
+C(1,\tfrac{\qq+1}{2})(t^{3/2}\|\nabla\alpha\|_{L_2(0,T;L_4(\Omega))})^{\frac{\qq+1}{\qq-1}}\\
&\quad+\|\alpha\|_{C(0,T;L_\infty(\Omega))} \Bigl(|\nabla u_{1,n}|_{L^2(\Omega)} |\nabla u_{0,n}|_{L_2(\Omega)}
+\|\nabla \psii_{n,t}\|_{L^2(0,t;L_2(\Omega))}^2\Bigr)\\
&\quad+\tfrac{b(1-\delta)}{4}|\nabla \psii_{n,t}(t)|_{L^2(\Omega)}^2
+\tfrac{1}{b(1-\delta)}\|\alpha\|_{C(0,T;L_\infty(\Omega))}^2|\nabla \psii_{n}(t)|_{L^2(\Omega)}^2\,,
\end{align*}
which, by smallness of $\alpha$, implies
\begin{equation}\label{Wpot_visc_lin_enid1_iii_1}
\begin{aligned}
&\int_0^t c_2|\psii_{n,tt}(s)|_{L_2(\Omega)}^2\, ds
+\left[\tfrac{b(1-\delta)}{4}|\nabla\psii_{n,t}|_{L_2}^2
+\tfrac{b\delta}{\qq+1}|\nabla\psii_{n,t}|_{L_{\qq+1}(\Omega)}^{\qq+1}
\right]_0^t\\
&\leq
(\|\alpha_t\|_{L^{4/3}(0,T;L_2(\Omega))}
+\|\nabla\alpha\|_{L_2(0,T;L_4(\Omega))}) \tfrac12 |\nabla u_{0,n}|_{L_4(\Omega)}^2\\
&\quad+\|\alpha\|_{C(0,T;L_\infty(\Omega))}
\Bigl(|\nabla u_{0,n}|_{L_2(\Omega)}^2+|\nabla u_{1,n}|_{L^2(\Omega)} |\nabla u_{0,n}|_{L_2(\Omega)}\\
&\quad+\|\nabla \psii_{n,t}\|_{L^2(0,t;L_2(\Omega))}^2\Bigr)+ 2 \|\nabla \psii_{n,t}\|_{L^4(0,T;L_4(\Omega))})^{\qq+1}\\
&\quad+C(1,\tfrac{\qq+1}{2})\|\alpha_t\|_{L^{4/3}(0,T;L_2(\Omega))}^{\frac{\qq+1}{\qq-1}}
+C(1,\tfrac{\qq+1}{2})(t^{3/2}\|\nabla\alpha\|_{L_2(0,T;L_4(\Omega))})^{\frac{\qq+1}{\qq-1}}\\
&\quad+\tfrac{1}{b(1-\delta)}\|\alpha\|_{C(0,T;L_\infty(\Omega))}^2|\nabla \psii_{n}(t)|_{L^2(\Omega)}^2
\end{aligned}
\end{equation}
for $c_2=1-\|\nabla\alpha\|_{L^2(0,T;L^4(\Omega))}$.
Adding \eqref{Wpot_visc_lin_enid0_iii_1} and $\epsilon$ times \eqref{Wpot_visc_lin_enid1_iii_1} with $\epsilon_0,\epsilon$ sufficiently small
\begin{align*}
&\epsilon<\min\left\{\frac{b(1-\delta)}{2\|\alpha\|_{C(0,T;L_\infty(\Omega))}},
\frac{c_1\ul{\alpha}b(1-\delta)}{2\|\alpha\|_{C(0,T;L_\infty(\Omega))}^2},
\frac{b\delta}{4({C_{L_{\qq+1},L_4}^\Omega})^{\qq+1}}\right\}\,,
\\
&\epsilon_0<\frac{b\delta\epsilon}{(\qq+1)({C_{L_{\qq+1},L_4}^\Omega})^{\qq+1}},
\end{align*}
where we use $\qq+1\geq 4$, yields the energy estimate
\begin{align*}
&\tfrac12\|\psii_{n,t}\|_{C(0,t;L_2(\Omega))}^2+c_1\|\sqrt{\alpha}\nabla\psii_{n}\|_{C(0,t;L_2(\Omega))}^2
+\tfrac{b(1-\delta)}{2}\|\nabla\psii_{n,t}\|_{L^2(0,t;L_2(\Omega))}^2 \\
&\quad+c_3\|\nabla\psii_{n,t}\|_{L^{\qq+1}(0,t;L_{\qq+1}(\Omega))}^{\qq+1}
+\epsilon c_2\|\psii_{n,tt}\|_{L^2(0,t;L_2(\Omega))}^2\\
&\quad+\epsilon \tfrac{b(1-\delta)}{2} \|\nabla\psii_{n,t}\|_{C(0,t;L_2(\Omega))}^2
+\epsilon \tfrac{b\delta}{\qq+1}|\nabla\psii_{n,t}|_{C(0,t;L_{\qq+1}(\Omega))}^{\qq+1}\\
&\leq
C(\epsilon_0,\tfrac{\qq+1}{2})(\tfrac12\|\nabla\alpha\|_{L_2(0,T;L_4(\Omega))})^{\frac{\qq+1}{\qq-1}}\\
&\quad+\tfrac12|u_{1,n}|_{L_2(\Omega)}^2+\tfrac12 \|\alpha\|_{C(0,T;L^\infty(\Omega))}|\nabla u_{0,n}|_{L_2(\Omega)}^2\\
&\quad+\epsilon(\|\alpha_t\|_{L^{4/3}(0,T;L_2(\Omega))}
+\|\nabla\alpha\|_{L_2(0,T;L_4(\Omega))})\tfrac12 |\nabla u_{0,n}|_{L_4(\Omega)}^2\\
&\quad+\epsilon\|\alpha\|_{C(0,T;L_\infty(\Omega))}
\Bigl(|\nabla u_{0,n}|_{L_2(\Omega)}^2+|\nabla u_{1,n}|_{L^2(\Omega)} |\nabla u_{0,n}|_{L_2(\Omega)}\Bigr)\\
&\quad+\epsilon C(1,\tfrac{\qq+1}{2})\|\alpha_t\|_{L^{4/3}(0,T;L_2(\Omega))}^{\frac{\qq+1}{\qq-1}}
+\epsilon C(1,\tfrac{\qq+1}{2})(t^{\frac{3}{2}} \|\nabla\alpha\|_{L_2(0,T;L_4(\Omega))})^{\frac{\qq+1}{\qq-1}}\\
&\quad+\epsilon \tfrac{b(1-\delta)}{2} |\nabla\psii_{1,n}|_{L_2(\Omega)}^2
+\epsilon \tfrac{b\delta}{\qq+1}|\nabla\psii_{1,n}|_{L_{\qq+1}(\Omega)}^{\qq+1},
\end{align*}
for some $c_3>0,$ which leads to \eqref{enest_Wpot_viscosity_lin2}.
\end{proof}
\begin{remark}
Note that multiplication with $\psii_t$ (via $q\geq d-1$) according to \eqref{enest_Wpot_viscosity_lin0} only gives an
$L_2(0,T;L_\infty(\Omega))$ bound on $\psii_{t}$ and in order to obtain a $C(0,T;L_\infty(\Omega))$ bound,
multiplication with $\psii_{tt}$ is required, cf. \eqref{estLinftyW3}.
Thus part (i) of Proposition \ref{prop:Wpot_viscosity_lin} is only an intermediate result.
\\
In part (ii) we make use of a positive sign of the term pertaining to the potential energy in the equation.
However this leads to an $L_4(\Omega)$ norm term on the right hand side of the energy identity, which we can
only dominate by means of the $L_{\qq+1}(\Omega)$ norm term on the left hand side. Hence, for this part, $\qq+1\geq 4$ is needed.
\end{remark}
\begin{theorem}\label{th:W3locex}
Let $c^2,b,\delta,1-\delta>0$, $k\in\R$, $\qq\geq3$.\\
There exist $\kappa>0$, $\bar{m}>0$, $\bar{M}>0$, $T>0$  such that for all $\psii_0 \in H_0^1(\Omega)$, $u_1\in H_0^1(\Omega) \cap W^{1,\qq+1}(\Omega)$ with
\begin{equation}\label{kappaW3}
|\psii_1|_{L_2(\Omega)}^2 +|\nabla \psii_0|_{L_2(\Omega)}^2
+ |\nabla \psii_1|_{L_2(\Omega)}^2
+ |\nabla \psii_1|_{L_{\qq+1}(\Omega)}^{\qq+1}
\leq \kappa^2
\end{equation}
there exists a weak solution $\psii\in \cW\subseteq X=H^2(0,T;L_2(\Omega))\cap C^1(0,T;W^{1,\qq+1}(\Omega))$ of \eqref{Wpot_viscosity} where
\begin{eqnarray}\label{defcW_W3}
\cW =\{v\in X
&:& \|v_{tt}\|_{L_2(0,T;L_2(\Omega))}\leq \bar{m}\nonumber\\
&& \wedge \|\nabla v_t\|_{C(0,T;L_2(\Omega))}\leq \bar{m}\nonumber\\
&& \wedge \|\nabla v_t\|_{C(0,T;L_{\qq+1}(\Omega))}\leq \bar{M}\}\,.
\end{eqnarray}
\end{theorem}
\begin{proof}
Relying on Proposition \ref{prop:Wpot_viscosity_lin}, we carry out the proof by means of a fixed point argument.
To this end, we define the fixed point operator $\cT :\cW \to X$, $v\mapsto \cT v=\psii$ where $\psii$ solves \eqref{Wpot_viscosity_lin} with
\begin{equation}\label{alpha_Wpotviscosity}
\alpha=\frac{c^2}{1-2\tilde{k}v_t}\,,
\end{equation}
which is well-defined by Proposition \ref{prop:Wpot_viscosity_lin} (ii), since we have
\[
\begin{aligned}
&\alpha(t,x)\geq\frac{c^2}{1+2\tilde{k}\|v_t\|_{C(0,T;L_\infty(\Omega))}}
\geq \frac{c^2}{1+2\tilde{k}C_{W_0^{1,\qq+1},L_\infty}^\Omega\bar{M}}\geq \frac{2c^2}{3}\\
&\|\alpha\|_{C(0,T;L_\infty(\Omega))}\leq\frac{c^2}{1-2\tilde{k}\|v_t\|_{C(0,T;L_\infty(\Omega))}}
\leq \frac{c^2}{1-2\tilde{k}C_{W_0^{1,\qq+1},L_\infty}^\Omega\bar{M}}\leq 2c^2\\
&\|\nabla\alpha\|_{L_2(0,T;L_4(\Omega))}=\|\frac{2\tilde{k}c^2}{(1-2\tilde{k}v_t)^2}\nabla v_t\|_{L_2(0,T;L_4(\Omega))}\\
&\quad \leq \frac{2\tilde{k}c^2}{(1-2\tilde{k}C_{W_0^{1,\qq+1},L_\infty}^\Omega\bar{M})^2}
 C_{L_{\qq+1},L_4}^\Omega \sqrt{T}\bar{M}
\leq \tilde{k}c^2 C_{L_{\qq+1},L_4}^\Omega \sqrt{T}\bar{M}
\\
&\|\alpha_t\|_{L^{4/3}(0,T;L_2(\Omega))}=\|\frac{2\tilde{k}c^2}{(1-2\tilde{k}v_t)^2} v_{tt}\|_{L^{4/3}(0,T;L_2(\Omega))}\\
&\quad \leq \frac{2\tilde{k}c^2}{(1-2\tilde{k}C_{W_0^{1,\qq+1},L_\infty}^\Omega\bar{M})^2}
\sqrt[4]{T} \bar{m}
\leq 8\tilde{k}c^2 \sqrt[4]{T} \bar{m}
\,,
\end{aligned}
\]
where we have used $v\in\cW$, so the positivity and smallness assumptions on $\alpha$ in Proposition \ref{prop:Wpot_viscosity_lin} are satisfied, provided $\bar{m}$, $\bar{M}$ are sufficiently small, in particular
$\bar{M}\leq (4\tilde{k}C_{W_0^{1,\qq+1},L_\infty}^\Omega)^{-1}$.
\\
Hence, the energy estimate \eqref{enest_Wpot_viscosity_lin2} yields
\begin{equation*}
\begin{aligned}
&\tfrac12\|\psii_{n,t}\|_{C(0,T;L_2(\Omega))}^2+\|\sqrt{\alpha}\nabla\psii_{n}\|_{C(0,T;L_2(\Omega))}^2
+\|\nabla\psii_{n,t}\|_{L^2(0,T;L_2(\Omega))}^2 \\
&+\|\nabla\psii_{n,t}\|_{L^{\qq+1}(0,T;L_{\qq+1}(\Omega))}^{\qq+1}
\\
&+\|\psii_{n,tt}\|_{L^2(0,T;L_2(\Omega))}^2
+\|\nabla\psii_{n,t}\|_{C(0,T;L_2(\Omega))}^2
+|\nabla\psii_{n,t}|_{C(0,T;L_{\qq+1}(\Omega))}^{\qq+1}\\
\leq &
\bar{C}\Bigl(
(\sqrt{T}\bar{M})^{\frac{\qq+1}{\qq-1}}
+|u_1|_{L_2(\Omega)}^2+ |\nabla u_0|_{L_2(\Omega)}^2\\
&+(\sqrt[4]{T}\bar{m}+\sqrt{T}\bar{M}) |\nabla u_0|_{L_4(\Omega)}^2\\
&+|\nabla u_0|_{L_2(\Omega)}^2+|\nabla u_1|_{L^2(\Omega)} |\nabla u_0|_{L_2(\Omega)}\\
&+(\sqrt[4]{T}\bar{m})^{\frac{\qq+1}{\qq-1}}
+(T^2\bar{M})^{\frac{\qq+1}{\qq-1}}\\
&+|\nabla\psii_1|_{L_2(\Omega)}^2
+|\nabla\psii_1|_{L_{\qq+1}(\Omega)}^{\qq+1}
\Bigr)\,.
\end{aligned}
\end{equation*}
Thus by making $T$ and the bound $\kappa$ on the initial data sufficiently small, by an appropriate choice of
$\bar{m}$, $\bar{M}$ we can achieve that $\psii\in\cW$. By closedness of $\cW$ we obtain existence of a solution.
\\
\end{proof}
\begin{remark}\label{rem:Wpot_viscosity_uniqueness}
To see that contractivity and therefore also uniqueness fails for \eqref{Wpot_viscosity}, consider two solutions $\psii^i=\cT(v^i)$, $i=1,2$
as well as their difference $\hat{u}=\psii^1-\psii^2$, which is a weak solution to
\begin{equation}\label{Wpot_vicosity_lin_diff}
\begin{cases}
\hat{u}_{tt}-\frac{c^2}{1-2\tilde{k}v^1_t}\Delta \hat{u}-b(1-\delta)\Delta \hat{u}_t b\delta\int_0^1\,\text{div}\Bigl( w^\sigma
\Bigl[|\nabla (u^2+\sigma \hat{u})_t|^2  \nabla \hat{u}_t\\
+(\qq-1)(\nabla (u^2+\sigma \hat{u})_t \nabla \hat{u}_t)
\nabla(u^2+\sigma \hat{u})_t \Bigr]\,  \Bigr)d\sigma
=\frac{2\tilde{k}c^2}{(1-2\tilde{k}v^1_t)(1-2\tilde{k}v^2_t)}\Delta v^2\hat{v}_t
\\
(\hat{u},\hat{u}_t)|_{t=0}=(0,0),
\\
\hat{u}|_{\partial \Omega} =0
,\\
\end{cases}
\end{equation}
with $\hat{v}=v^1-v^2$ and $w^\sigma(x,t)=|\nabla(u^2+\sigma\hat{u})_t(x,t)|^{q-3}$ as in \eqref{wsig}.
Upon multiplication with $\hat{u}_t$ and integration with respect to space and time, like in the proof of Proposition \ref{prop:W2lin} and
Theorem \ref{th:W2locex}, the $b\delta$ term yields a nonnegative contribution on the left hand side.
Therewith, similarly to \eqref{Wpot_visc_lin_enid0_iii}, we obtain
\[
\begin{aligned}
&\tfrac12\left[|\hat{u}_t|_{L_2(\Omega)}^2+|\frac{c}{\sqrt{1-2\tilde{k}v^1_t}}\nabla\hat{u}|_{L_2(\Omega)}^2\right]_0^t
+\int_0^t b(1-\delta)|\nabla\hat{u}_t|_{L_2(\Omega)}^2
\\
=&\int_0^t\int_\Omega\Bigl( -\frac{\tilde{k}c^2}{(1-2\tilde{k}v^1_t)^2} v^1_{tt} |\nabla \hat{u}|^2
+\hat{u}_t \frac{2\tilde{k}c^2}{(1-2\tilde{k}v^1_t)^2}\nabla v^1_t \nabla \hat{u}\\
&\quad
-\frac{2\tilde{k}c^2}{(1-2\tilde{k}v^1_t)(1-2\tilde{k}v^2_t)}\nabla v^2 \nabla\hat{v}_t
-\frac{4\tilde{k}^2c^2}{(1-2\tilde{k}v^1_t)^2(1-2\tilde{k}v^2_t)}\nabla v^1_t\nabla v^2 \hat{v}_t\\
&\quad
-\frac{4\tilde{k}^2c^2}{(1-2\tilde{k}v^1_t)(1-2\tilde{k}v^2_t)^2}\nabla v^2_t\nabla v^2 \hat{v}_t
\Bigr) \, dx \, ds
\end{aligned}
\]
We see that estimation of the first term on the right hand side (the one containing $v^1_{tt} |\nabla \hat{u}|^2$) would require higher spatial summability of $v^1_{tt}$ and/or of $\nabla \hat{u}$. The estimates on these two quantities, following from $v^1\in\cW$ and from the left hand side of the above estimate on $\hat{u}$ are not strong enough for this purpose. However, multiplication of \eqref{Wpot_viscosity} with higher time or space derivatives of $\psii$ leads to difficulties involving the strong nonlinear damping term; the same holds for multiplication of \eqref{Wpot_vicosity_lin_diff} with $\hat{u}_{tt}$.
\end{remark}

\section{Linear wave -- nonlinear Westervelt equation coupling via nonlinear strong damping \eqref{Wacoustic_viscosity}}
\label{secWacoustic_viscosity}
We make the following assumptions on the spatially varying coefficients
$\lambda$, $k$, $\varrho$, $b$, $\delta$ in \eqref{Wacoustic_viscosity}:

\begin{equation}\label{condcoeffacoust}
\begin{cases}
k, \lambda, \varrho, b\geq0, \delta \in L_\infty(\Omega), 0 \leq\delta\leq 1,\\
\exists \ul{\lambda},\ul{\varrho}>0: \
\ul{\lambda}\leq \lambda(x) \,, \ \ul{\varrho}\leq\varrho(x)\
\mbox{ in }\Omega,\\
\exists \ul{b},\ul{\delta}>0: \
\ul{b}\leq b(x) \,, \ \ul{\delta}\leq\delta(x)\
\mbox{ in }\Omega_{nl},\\
\mbox{where } \Omega_{nl}:=\{x\in\Omega\, : \, k(x)\not=0\}.
\end{cases}
\end{equation}
Since \eqref{Wacoustic_viscosity} can be viewed as a version of \eqref{Wpress_viscosity} with spatially varying coefficients, inspection of the proof of Theorem \ref{th:W2locex} immediately yields

\begin{corollary}
Let $\qq>d-1$ and suppose assumptions \eqref{condcoeffacoust} are satisfied.\\
For any $T>0$ there is a $\kappa_T>0$ such that for all $u_0,u_1\in H_0^1(\Omega)\cap W^{1,\qq+1}(\Omega)$ with
$$E_1[u](0)+|\nabla u_0|_{L_{\qq+1}}^2 \leq \kappa_T^2\,,$$
there exists a weak solution $u\in \cW $ (defined as in \eqref{defcW} with
\eqref{smallnessMbar} and $\bar{m}$ sufficiently small) of \eqref{Wacoustic_viscosity} and $u$ is unique in $\cW $.
\end{corollary}
\begin{remark}\label{rem:interfacecond}
Consider the case that the coefficients $\lambda, k, \varrho, b, \delta$ are piecewise constant, for simplicity of exposition, only on two open subdomains $\Omega_I$, $\Omega_{II}$ with $\Omega_I\cap\Omega_{II}=\emptyset$, $\Omega_I\cup\Omega_{II}=\Omega$, with a smooth interface $\Gamma$ defined by $\overline{\Gamma}=\ol{\Omega}_I\cap\ol{\Omega}_{II}$.
Due to the fact that (cf., e.g., \cite{BrezziFortin}, \cite{RaviartThomas})
$$
H^1(\Omega)= \{v\in L^2(\Omega)\, : \, v\vert_{\Omega_i}=v_i\in H^1(\Omega_i)
\,, \, i=I,II\,, \ \mbox{tr}_\Gamma v_I=\mbox{tr}_\Gamma v_{II}\}\,,
$$
multiplying the PDE \eqref{Wacoustic_viscosity} with $v\in H^1(\Omega)$ and integrating by parts separately on each subdomain (cf. equations (2.3), (2.5) in \cite{BambergerGlowinskiTran}),  we arrive at the interface conditions
\begin{eqnarray*}
&&\mbox{tr}_\Gamma (u\vert_{\Omega_I})=\mbox{tr}_\Gamma (u\vert_{\Omega_{II}})\\
&&\mbox{ i.e., continuity of pressure in a trace sense}\\
&&\int_\Gamma \Bigl\{\mathbf{v}_u\vert_{\Omega_I}-\mathbf{v}_u\vert_{\Omega_{II}}
\Bigr\} \nu \, v \,ds=0 \quad \forall v\in H^1(\Omega)\\
&& \mbox{where }\mathbf{v}_u:= \frac{1}{\varrho} \nabla u
+ b\,\Bigl((1-\delta) +\delta|\nabla u_t|^{\qq-1}\Bigr)\nabla u_t\,,\\
&&\mbox{ i.e., continuity of (modified) normal velocity in a variational sense.}
\end{eqnarray*}
Here the notation for velocity is motivated by the linearized Euler equation
$\varrho\mathbf{v}_t=-\nabla p$ relating the velocity $\mathbf{v}$ with the pressure $p$, which is denoted by $u$ here.
\end{remark}

\section{Linear elasticity -- nonlinear Westervelt equation coupling via nonlinear strong damping \eqref{Welastic_viscosity}}
\label{secWelastic_viscosity}

In this section we consider the coupled nonlinear acoustic -- elastic system
\begin{equation}\label{Welastic_viscosity1}
\begin{cases}
\varrho(x) \bu_{tt}-\cB^T\frac{1}{1-2\tilde{k}(x)\psi_t} [c](x)\cB\bu\\
\qquad \qquad+\cB^T\Bigl(((1-\delta(x)) +\delta(x)|\cB \bu_t|^{\qq-1})[b](x)\cB \bu_t\Bigr)
=0,
\\
(\bu,\bu_t)|_{t=0}=(u_0, u_1),
\\
\bu|_{\partial \Omega} =0
,\\
-\Delta \psi=-\text{div}\, \bu
\psi|_{\partial \Omega} =0.
\end{cases}
\end{equation}
We again assume that $\Omega\subseteq\R^d$ is an open bounded set with Lipschitz boundary and
impose the following conditions on the coefficients in \eqref{Welastic_viscosity1}
\begin{equation}\label{condcoeffelast}
\begin{cases}
\tilde{k}, \varrho, [b], [c], \delta \in L_\infty(\Omega), ~0\leq\delta\leq 1,\\
\hfill \mbox{ with }[b(x)], [c(x)]\mbox{ symmetric positive semidefinite matrices},\\
\exists \ul{\varrho},\ol{\varrho},\ul{c},\ol{c},\ul{\gamma}>0: \
\quad \ul{\varrho}\leq \varrho(x)\leq \ol{\varrho}, \\
\qquad \qquad \qquad \qquad \quad \ul{c}\leq \lambda_{\min}([c(x)])\leq \lambda_{\max}([c(x)])\leq \ol{c},\\
\qquad \qquad \qquad \qquad \quad \ul{\gamma}\leq (1-\delta(x))\lambda_{\min}([b(x)]),\quad\hfill \mbox{ in }\Omega, \\
\exists \ul{\delta},\ul{b},\ol{b}>0\ \exists \ol{\delta}<1: \
\ul{b}\leq \lambda_{\min}([b(x)])\leq \lambda_{\max}([b(x)])\leq\ol{b},\\
\qquad \qquad \qquad \qquad \quad \ul{\delta}\leq\delta(x)\leq\ol{\delta}, \quad \hfill \mbox{ in }\Omega_{nl},\\
 \qquad\qquad\mbox{where } \Omega_{nl}:=\{x\in\Omega\, : \, \tilde{k}(x)\not=0\}\mbox{ is an open domain and}\\
 \qquad\qquad \mbox{$\overline{\Omega_{nl}}$ is a compact subset of $\Omega$ or $\Omega$ has $C^2$ boundary}\,.
\end{cases}
\end{equation}
In particular, strong linear damping is required in the whole domain, whereas strong nonlinear damping is only needed in the region $\Omega_{nl}$ of nonlinear acoustics.
\\
Similarly to Theorem \ref{th:W3locex} we obtain the following local existence result
\begin{theorem}\label{th:W5locex}
Let $\qq\geq3$ and
assumption \eqref{condcoeffelast}
be satisfied.\\
There exist $\kappa>0$, $\bar{m}>0$, $\bar{a}>0$, $\bar{M}>0$, $T>0$ such that for all $\bu_0,\bu_1\in (H_0^1(\Omega))^d$, $\bu_1\vert_{\Omega_{nl}}\in (W^{1,\qq+1}(\Omega_{nl}))^d$ with
\[
|\bu_1|_{L_2(\Omega)}^2 +|\cB \bu_0|_{L_2(\Omega)}^2
+ |\cB \bu_1|_{L_2(\Omega)}^2
+ |\cB \bu_1|_{L_{\qq+1}(\Omega_{nl})}^{\qq+1}
\leq \kappa^2
\]
we have existence of a weak solution $\bu\in \cW\subseteq H^2(0,T;(L_2(\Omega))^d)\cap C^1(0,T;(H_0^1(\Omega))^d)$ of \eqref{Welastic_viscosity1} where
\begin{equation}\label{defcW_W5}
\begin{aligned}
\cW =\{\bv\in X
:~~& \|\bv_{tt}\|_{L_2(0,T;L_2(\Omega))}\leq \bar{m}\\
& \wedge \|\cB \bv_t\|_{C(0,T;L_2(\Omega))}\leq \bar{a}\\
& \wedge \|\cB \bv_t\|_{C(0,T;L_{q+1}(\Omega_{nl}))}\leq \bar{M}\}\,.
\end{aligned}
\end{equation}
\end{theorem}
\begin{proof}
Recall that $\psi$ is related to $\bu$ via $-\Delta \psi=-\text{div}\,\bu$ and that we impose homogeneous Dirichlet boundary conditions on $\psi$.
Like in the proof of Theorem \ref{th:W3locex} we can use a fixed point argument with an operator $\cT$ mapping $\bv$ to a weak solution of
\begin{equation}\label{Welastic_viscosity_fp}
\begin{cases}
\varrho(x) \bu_{tt}-\cB^T\alpha(t,x) [c](x)\cB\bu\\
\qquad \qquad +\cB^T\Bigl(((1-\delta(x)) +\delta(x)|\cB \bu_t|^{\qq-1})[b](x)\cB \bu_t\Bigr)
=0,
\\
(\bu,\bu_t)|_{t=0}=(u_0, u_1),
\\
\bu|_{\partial \Omega} =0,
\\
\end{cases}
\end{equation}
where
\[
\alpha(t,x)=\frac{1}{1+2\tilde{k}(x)(-\Delta)^{-1}\text{div}\,\bv_t(t,x)}\,,
\]
$-\Delta$ denotes the Laplace operator on $\Omega$ with homogeneous Dirichlet boundary conditions on $\partial \Omega$, and
$C_{\Delta,\Omega_{nl}}$ is the constant in the regularity estimate
\[
\|(-\Delta)^{-1} f\|_{H^2(\Omega_{nl})}\leq C_{\Delta,\Omega_{nl}} \|f\|_{L_2(\Omega)}\,.
\]
Existence of a weak solution for fixed $\bv\in\cW$ can be shown like in Proposition \ref{prop:Wpot_viscosity_lin}.
In particular by choosing $\bar{a}$ sufficiently small so that
\[
2\|\tilde{k}\|_{L_\infty(\Omega_{nl})} C_{\Delta,\Omega_{nl}} C_{H^2(\Omega_{nl}),L_\infty(\Omega_{nl})}\bar{a}
< 1
\]
we have that
\[
\begin{aligned}
\alpha(t,x)
\geq&\frac{1}{1+ 2\|\tilde{k}\|_{L_\infty(\Omega_{nl})} C_{\Delta,\Omega_{nl}} C_{H^2(\Omega_{nl}),L_\infty(\Omega_{nl})}\bar{a}}=:\ul{\alpha}>0\\
\alpha(t,x)
\leq&\frac{1}{1- 2\|\tilde{k}\|_{L_\infty(\Omega_{nl})} C_{\Delta,\Omega_{nl}} C_{H^2(\Omega_{nl}),L_\infty(\Omega_{nl})}\bar{a}}=:\ol{\alpha}<\infty\\
|\alpha_t(t,x)|=&
\left|\frac{2\tilde{k}(x)}{(1+2\tilde{k}(x)(-\Delta)^{-1}\text{div}\,\bv_t(t,x))^2}
(-\Delta)^{-1}\text{div}\,\bv_{tt}(t,x)\right|\\
\leq& \frac{2\|\tilde{k}\|_{L_\infty(\Omega_{nl})}|(-\Delta)^{-1}\text{div}\,\bv_{tt}(t,x)|
}{1- 2\|\tilde{k}\|_{L_\infty(\Omega_{nl})} C_{\Delta,\Omega_{nl}} C_{H^2(\Omega_{nl}),L_\infty(\Omega_{nl})}\bar{a}} \quad \mbox{ on }\Omega_{nl}\\
\alpha_t(t,x)=&\ 0 \mbox{ on }\Omega\setminus\Omega_{nl}\\
\|\alpha_t\|_{L^p(0,T;L_2(\Omega))}\leq&
\frac{2\|\tilde{k}\|_{L_\infty(\Omega_{nl})} \|\nabla (-\Delta)^{-1}\text{div}\,\bv_{tt}\|_{L^p(0,T;L_2(\Omega))}
}{1-2\|\tilde{k}\|_{L_\infty(\Omega_{nl})} C_{\Delta,\Omega_{nl}} C_{H^2(\Omega_{nl}),L_\infty(\Omega_{nl})}\bar{a}} \\
\leq&
\frac{2\|\tilde{k}\|_{L_\infty(\Omega_{nl})} T^{\frac{2-p}{2p}}\bar{m}
}{1-2\|\tilde{k}\|_{L_\infty(\Omega_{nl})} C_{\Delta,\Omega_{nl}} C_{H^2(\Omega_{nl}),L_\infty(\Omega_{nl})}\bar{a}} \\
=:&\,  \ol{\ol{\alpha}} T^{\frac{2-p}{2p}} \bar{m}\,, \quad p\in[1,2]
\end{aligned}
\]
for $\bv\in\cW$, provided $\bar{a}$ is chosen sufficiently small.
By multiplying \eqref{Welastic_viscosity_fp} with
$\bu_t+\varepsilon\bu_{tt}$
we obtain
\begin{align*}
&\tfrac12\left[|\sqrt{\varrho}\bu_t|_{L_2(\Omega)}^2\right]_0^t
+\tfrac12\left[|\sqrt{\alpha [c]}\cB\bu|_{L_2(\Omega)}^2 \right]_0^t +\varepsilon\int_0^t|\sqrt{\varrho}\bu_{tt}|_{L_2(\Omega)}^2\, ds\\
&\quad+\int_0^t \Bigl(|\sqrt{(1-\delta)[b]}\cB\bu_t|_{L_2(\Omega)}^2
+|\sqrt{\delta[b]}\cB\bu_t|_{L_{\qq+1}(\Omega_{nl})}^{\qq+1}\Bigr)\, ds\\
&\quad
+\varepsilon\left[\tfrac12|\sqrt{(1-\delta)[b]}\cB\bu_t|_{L_2(\Omega)}^2
+\tfrac{1}{\qq+1}|\sqrt[\qq+1]{\delta[b]}\cB\bu_t|_{L_{\qq+1}(\Omega_{nl})}^{\qq+1}\right]_0^t\\
&=
\tfrac12\int_0^t\int_{\Omega_{nl}}\Bigl(\alpha_t[c]\cB\bu\Bigr)^T\cB\bu\, dx\, ds\\
&\quad+\varepsilon\int_0^t |\sqrt{\alpha [c]}\cB\bu_t(s)|_{L_2(\Omega)}^2\, ds
+\varepsilon\int_0^t\int_{\Omega_{nl}}\Bigl(\alpha_t[c]\cB\bu\Bigr)^T\cB\bu_t\, dx\, .ds\\
&\quad-\varepsilon\left[ \int_\Omega \Bigl(\alpha[c]\cB\bu\Bigr)^T\cB\bu_t\, dx\right]_0^t.
\end{align*}
We estimate the right-hand side and arrive at
\begin{align*}
&\tfrac12\left[|\sqrt{\varrho}\bu_t|_{L_2(\Omega)}^2\right]_0^t
+\tfrac12\left[|\sqrt{\alpha [c]}\cB\bu|_{L_2(\Omega)}^2 \right]_0^t+\varepsilon\int_0^t|\sqrt{\varrho}\bu_{tt}|_{L_2(\Omega)}^2\, ds \\
&\quad+\int_0^t \Bigl(|\sqrt{(1-\delta)[b]}\cB\bu_t|_{L_2(\Omega)}^2
+|\sqrt{\delta[b]}\cB\bu_t|_{L_{\qq+1}(\Omega_{nl})}^{\qq+1}\Bigr)\, ds\\
&\quad
+\epsilon\left[\tfrac12|\sqrt{(1-\delta)[b]}\cB\bu_t|_{L_2(\Omega)}^2
+\tfrac{1}{\qq+1}|\sqrt[\qq+1]{\delta[b]}\cB\bu_t|_{L_{\qq+1}(\Omega_{nl})}^{\qq+1}\right]_0^t\\
&\leq
\tfrac12\ol{\ol{\alpha}} t^{\frac12} \bar{m}\ol{c}\int_0^t |\cB\bu(s)|_{L_4(\Omega_{nl})}^2 ds\\
&\quad+\varepsilon\ol{\alpha}\ol{c}\int_0^t |\cB\bu_t(s)|_{L_2(\Omega)}^2\, ds
+\varepsilon \ol{\ol{\alpha}} \bar{m}\ol{c}\int_0^t |\cB\bu(s)|_{L_4(\Omega_{nl})} |\cB\bu_t(s)|_{L_4(\Omega_{nl})} ds\\
&\quad+\varepsilon\ol{\alpha}\ol{c}\Bigl(|\cB\bu(t)|_{L_2(\Omega)} |\cB\bu_t(t)|_{L_2(\Omega)}
+|\cB\bu_0|_{L_2(\Omega)} |\cB\bu_1|_{L_2(\Omega)} \Bigr)\\
&\leq
\ol{\ol{\alpha}}\bar{m}\ol{c} \Bigl(t^{\frac32}|\cB\bu_0|_{L_4(\Omega_{nl})}+t^3\|\cB\bu_t\|_{L_4(0,T;L_4(\Omega_{nl}))}^2\Bigr)\\
&\quad+\varepsilon\ol{\alpha}\ol{c}\int_0^t |\cB\bu_t(s)|_{L_2(\Omega)}^2\, ds
+\varepsilon \ol{\ol{\alpha}}\bar{m}\ol{c}
\Bigl( \tfrac12 |\cB\bu_0|_{L_4(\Omega_{nl})}+\tfrac32 t^{\frac32}\|\cB\bu_t\|_{L_4(0,T;L_4(\Omega_{nl}))}^2\Bigr)\\
&\quad+\varepsilon\tfrac{\gamma}{4} |\cB\bu_t(t)|_{L_2(\Omega)}^2
+\varepsilon\tfrac{(\ol{\alpha}\ol{c})^2}{\gamma} |\cB\bu(t)|_{L_2(\Omega)}^2
+\varepsilon\ol{\alpha}\ol{c}|\cB\bu_0|_{L_2(\Omega)} |\cB\bu_1|_{L_2(\Omega)},
\end{align*}
thus by \eqref{condcoeffelast} using \eqref{abeps}, \eqref{Cepsr} in
\begin{align*}
&\ol{\ol{\alpha}}\bar{m}\ol{c} (t^3+\varepsilon \tfrac32 t^{\frac32}) \|\cB\bu_t\|_{L_4(0,T;L_4(\Omega_{nl}))}^2\\
&\leq \tfrac{\ul{\delta}\ul{b}}{4} \|\cB\bu_t\|_{L_{\qq+1}(0,T;L_{\qq+1}(\Omega_{nl}))}^{\qq+1}
+C(\tfrac{\ul{\delta}\ul{b}}{4},\tfrac{\qq+1}{2})
\Bigl(\ol{\ol{\alpha}}\bar{m}\ol{c} (t^3+\varepsilon \tfrac32 t^{\frac32}) t^{\frac{\qq-3}{2(\qq+1)}}({C_{L_{\qq+1},L_4}^\Omega})^2\Bigr)^{\frac{\qq+1}{\qq-1}}
\end{align*}
we get
\begin{equation}\label{enest_Welastic_viscosity}
\begin{aligned}
&\tfrac14\ul{\varrho}\|\bu_t\|_{C(0,T;L_2(\Omega))}^2
+(\tfrac14\ul{\alpha}\ul{c}-\tfrac{(\varepsilon\ol{\alpha}\ol{c})^2}{\ul{\gamma}})\|\cB\bu\|_{C(0,T;L_2(\Omega))}^2  +\tfrac{\varepsilon}{2}\ul{\varrho}\|\bu_{tt}|_{L_2(0,T;L_2(\Omega))}^2\\
&\quad+(\tfrac12\ul{\gamma}-\varepsilon\ol{\alpha}\ol{c})\|\cB\bu_t\|_{L^2(0,T;L_2(\Omega))}^2
+\tfrac{\ul{\delta}\ul{b}}{4}\|\cB\bu_t\|_{L_{\qq+1}(0,T;L_{\qq+1}(\Omega_{nl}))}^{\qq+1}\\
&\quad
+\tfrac{\varepsilon}{2}\Bigl(\tfrac12\ul{\gamma}\|\cB\bu_t\|_{C(0,T;L_2(\Omega))}^2
+\tfrac{1}{\qq+1}\ul{\delta}\ul{b}\|\cB\bu_t\|_{C(0,T;L_{\qq+1}(\Omega_{nl}))}^{\qq+1}\Bigr)\\
&\leq\varepsilon\ol{\alpha}\ol{c}|\cB\bu_0|_{L_2(\Omega)} |\cB\bu_1|_{L_2(\Omega)}+\varepsilon\ol{\alpha}\ol{c}|\cB\bu_0|_{L_2(\Omega)} |\cB\bu_1|_{L_2(\Omega)}  \\
&\quad +\ol{\varrho} |\bu_1|_{L_2(\Omega)}^2 +\ol{\alpha} |\cB\bu_0|_{L_2(\Omega)}^2
+\varepsilon\Bigl(\ol{b}|\cB\bu_1|_{L_2(\Omega)}^2
+\tfrac{2}{\qq+1}\ol{b}|\cB\bu_1|_{L_{\qq+1}(\Omega_{nl})}^{\qq+1}\Bigr)\\
&\quad+C(\tfrac{\ul{\delta}\ul{b}}{4},\tfrac{\qq+1}{2})
\Bigl(\ol{\ol{\alpha}}\bar{m}\ol{c} (t^3+\varepsilon \tfrac32 t^{\frac32}) t^{\frac{\qq-3}{2(\qq+1)}}({C_{L_{\qq+1},L_4}^\Omega})^2\Bigr)^{\frac{\qq+1}{\qq-1}}
\end{aligned}
\end{equation}
Thus by choosing $\varepsilon,\kappa$ sufficiently small, we obtain $\bu\in\cW$ .
\end{proof}
Due to the appearance of positive powers of $t$ in the energy estimate \eqref{enest_Welastic_viscosity}, also here only a local in time well-posedness result can be expected. Morever, for similar reasons as in Section \ref{secWpot_viscosity} uniqueness is not likely to hold here, see Remark \ref{rem:Wpot_viscosity_uniqueness}.

\begin{remark}\label{rem:interfacecond_ae}
Again we can derive an interface condition in the case that the coefficients $\lambda,  \mu, k, \varrho, b, \delta$ are piecewise
constant, i.e., in a setting similar to the one of Remark \ref{rem:interfacecond}. This yields
\begin{eqnarray*}
&&\mbox{tr}_\Gamma (\bu\vert_{\Omega_I})=\mbox{tr}_\Gamma (\bu\vert_{\Omega_{II}})\\
&&\mbox{ i.e., continuity of velocity in a trace sense}\\
&&\int_\Gamma \Bigl\{\sigma_u\vert_{\Omega_I}-\sigma_u\vert_{\Omega_{II}}
\Bigr\} \mathbf{\nu} \, \mathbf{v} \,ds=0 \quad \forall \mathbf{v}\in (H^1(\Omega))^d\\
&& \mbox{where }\sigma_u:= [c] \cB \bu
+ ((1-\delta(x)) +\delta(x)|\cB \bu_t|^{\qq-1})[b](x)\cB \bu_t\,,\\
&&\mbox{ i.e., continuity of (modified) normal stresses in a variational sense.}
\end{eqnarray*}
\end{remark}

\section{The Westervelt equation in acoustic pressure formulation with $p$-Laplace damping \eqref{Wpress_pLaplace}}
\label{secWpress_pLaplace}
Similarly to Section \ref{secWpress_viscosity}, we start with a result on the equation
\begin{equation}\label{W1lin}
\begin{cases}
(1+\alpha) u_{tt}-c^2\,\text{div}\Bigl(\nabla u+\vep|\nabla u|^{\pp-1}\nabla u\Bigr)-b\Delta u_t+fu_t=g
\\
(u,u_t)|_{t=0}=(u_0, u_1)
\\
u|_{\partial \Omega} =0
.\\
\end{cases}
\end{equation}

\begin{proposition} \label{prop:W1lin}
Let $T>0$, $c^2,\vep>0$, $b\geq0$ and assume that
\begin{itemize}
\item
$\alpha\in C(0,T;L_\infty(\Omega))$, $\alpha_t\in L_\infty(0,T;L_2(\Omega))$,
$-1<-\underline{\alpha}\leq \alpha(t,x)\leq \overline{\alpha}$,
\item
$f\in L_\infty(0,T;L_{2}(\Omega))$,
\item
$g\in L_2(0,T;L_{4/3}(\Omega))$,
\item
$u_0\in H_0^1(\Omega)$, $u_1\in L_2(\Omega)$.
\end{itemize}
with
$$\|f-\frac12\alpha_t\|_{L_\infty(0,T;L_{2}(\Omega))}\leq\bar{b}
<\frac{b}{(C_{H_0^1,L_4}^\Omega)^2}\,. $$
Then \eqref{W1lin} has a weak solution
\begin{align*}
u\in\tilde{X} := ~&C^1(0,T;L_2(\Omega))\, \cap \,  C(0,T;H_0^1(\Omega)) \\
&\cap C(0,T;W_0^{1,p+1}(\Omega))\,  \cap \, H^1(0,T;H_0^1(\Omega))
\end{align*}
which is unique in $\tilde{X}$ and satisfies the energy estimate
\begin{equation}
\label{W1estimate}
\begin{aligned}
&\left[\frac{1}{2}\int_\Omega (1+\alpha)(u_{t})^2\, dx
+ \frac{c^2}{2} |\nabla u|_{L_2(\Omega)}^2+ \frac{c^2 \vep}{p+1} |\nabla u|_{L_{p+1}(\Omega)}^{p+1} \right]_0^t\\
&\qquad+ \hat{b} \int_0^t |\nabla u_{t}|_{L_2(\Omega)}^2 \, ds \leq \frac{1}{4\check{b}} \|g\|_{L_2(0,T;L_{4/3}(\Omega))}^2.
\end{aligned}
\end{equation}
where $\hat{b}=b-(\bar{b}+\check{b})(C_{H^1,L^4}^\Omega)^2$ and $\check{b}= \frac{b}{(C_{H^1,L^4}^\Omega)^2}- \bar{b}$.\\
In particular, for the energy
\begin{equation}
\label{energy:W1}
\mathcal{E}[u](t):= \left[ |u_t|_{L_2(\Omega)}^2 + |\nabla u|_{L_2(\Omega)}^2 + |\nabla u |_{L_{p+1}(\Omega)}^{p+1}\right](t)
\end{equation}
we have
\begin{equation}\label{enest:W1}
\mathcal{E}[u](t) + c_1 \int_0^t |\nabla u_t|_{L_2(\Omega)}^2 \, ds \leq C_2 \left(\mathcal{E}[u](0) + \|g\|_{L_2(0,T;L_{4/3}(\Omega))}\right)
\end{equation}
for some constants $c_1$, $C_2$ only depending on $\alpha$, $c$, $b$, $\varepsilon$, $p$ and $C_{H^1,L^4}^\Omega$.
\end{proposition}
\begin{proof}
The proof can be done analogously to the one of Proposition \ref{prop:W2lin} 
by first showing that \eqref{W1linweak} below holds for smooth approximations $\alpha_k$, $f_k$
and $g_k$ of $\alpha$, $f$ and $g$ by means of Galerkin's method, energy estimates and weak limits
and then letting $\alpha_k\to \alpha$, $f_k \to f$ and $g_k \to g$, respectively. Here, we omit details and only
show the key energy estimates. The weak form of \eqref{W1lin} reads as
\begin{equation}
\label{W1linweak}
\begin{aligned}
\int_\Omega \Bigl\{(1+\alpha)u_{tt} w + c^2 \left( \nabla u + \vep |\nabla u|^{p-1} \nabla u\right) \nabla w
+b \nabla u_t \nabla w \Bigr\} \, dx&
\\
=\int_\Omega (g -fu_t) w\, dx\quad \forall w\in W_0^{1,p+1}(\Omega),&
\end{aligned}
\end{equation}
with initial conditions $(u_{0},u_{1})$. \\
Testing \eqref{W1linweak} with $w=u_t(t)$ and integrating with respect to time yields
\begin{eqnarray*}
&\lefteqn{\left[\frac{1}{2}\int_\Omega (1+\alpha)\left(u_{t}\right)^2\, dx
+ \frac{c^2}{2} |\nabla u|_{L_2(\Omega)}^2\right.}\\
&&\left.+ \frac{c^2 \vep}{p+1} |\nabla u|_{L_{p+1}(\Omega)}^{p+1} \right]_0^t  + b \int_0^t |\nabla u_{t}|_{L_2(\Omega)}^2 \, ds\\
&=&
- \int_0^t  \int_\Omega \left(f-\frac12\alpha_{t}\right) (u_{t})^2 \, dx \, ds
+ \int_0^t  \int_\Omega g u_{t} \, dx \, ds\\
&\leq&
(\bar{b}+ \check{b})\int_0^t | u_{t}|_{L_4(\Omega)}^2 \, ds +\frac{1}{4\check{b}} \|g\|_{L_2(0,T;L_{4/3}(\Omega))}^2,
\end{eqnarray*}
hence, by $\hat{b}=b-(\bar{b}+\check{b})(C_{H^1,L^4}^\Omega)^2>0$, we arrive at \eqref{W1estimate},
which directly implies \eqref{enest:W1}.
\end{proof}
\begin{remark}
Note that in the case of $p$-Laplace damping we do not obtain a higher order energy estimate
like in the proof of Proposition \ref{prop:W2lin}. Testing \eqref{W1linweak}
with $w=u_{tt}$ yields
\begin{align*}
\lefteqn{\int_0^t\int_\Omega \Bigl\{(1+\alpha)(u_{tt})^2 \, dx\, ds + \frac{b}{2} |\nabla u_{n,t}|^2 \big|_0^t}\nonumber\\
&=c^2 \int_0^t |\nabla u_{t}|_{L_2(\Omega)}^2 \, ds - c^2\left[ \int_\Omega \nabla u \nabla u_{t} \, dx\right]_0^t \nonumber\\
& \qquad- c^2 \vep \int_0^t \int_\Omega |\nabla u|^{p-1} \nabla u \nabla u_{tt} \, dx \, ds  - \int_0^t \int_\Omega (f u_{t}- g)u_{tt}\, dx \, ds,\nonumber
\end{align*}
but here the first term in the third line cannot be controlled.
\end{remark}

Using Proposition \ref{prop:W1lin} 
we now show local existence of solutions.
\begin{theorem}\label{th:W1locex}
Let $c^2,b,\vep>0$, $k\in\R$, $\pp>d-1$.\\
For any $T>0$ there is a sufficiently small
$\kappa_T>0$ such that for all $u_0 \in W_0^{1,p+1}(\Omega)$, $u_1\in L_2(\Omega)$ with
\begin{equation}\label{W1initialbound}
\mathcal{E}[u](0) = |u_1|_{L_2(\Omega)}^2+ |\nabla u_0|_{L_2(\Omega)}^2 + |\nabla u_0 |_{L_{p+1}(\Omega)}^{p+1} \leq \kappa_T^2
\end{equation}
there exists a weak solution $u\in \cW $ of \eqref{Wpress_pLaplace} where
\begin{align}\label{defcWplaplace}
\cW =\{v\in \tilde{X}
:~& \|v_{t}\|_{L_\infty(0,T;L_2(\Omega))}\leq \bar{m}\nonumber\\
& \wedge \|\nabla v_t\|_{L_2(0,T;L_2(\Omega))}\leq \bar{m}\nonumber\\
& \wedge \|\nabla v\|_{L_{\infty}(0,T;L_{p+1}(\Omega))}\leq \bar{M}\}
\end{align}
where $\bar{m}$ and $\bar{M}$ are sufficiently small and $2 k C_{W_0^{1,p+1}, L_\infty} \bar{M}<1$ is satisfied.
\end{theorem}
\begin{proof}
We define the fixed point operator $\mathcal{T}: \mathcal{W} \to \tilde{X}$, $v\mapsto \mathcal{T}v =u$ where $u$ is a solution of \eqref{W1lin} with
\begin{equation*}
\alpha=2k v\,, \quad f=2kv_t\,, \quad g=0\,,
\end{equation*}
which is well-defined by Proposition \ref{prop:W1lin}. First, note that
\begin{equation*}
\|f-\frac12\alpha_t\|_{L_\infty(0,T;L_{2}(\Omega)}\leq k \bar{m}
\end{equation*}
and, that by assumption we have $\alpha \in C(0,T; L_\infty(\Omega))$, $-1<-\underline{\alpha} \leq \alpha(t,x) \leq \bar{\alpha}$,
with $\underline{\alpha} = \bar{\alpha} = 2k C_{W_0^{1,p+1}, L_\infty} \bar{M}<1$. \\
Hence, we can make use of the energy estimate \eqref{W1estimate} to conclude
\begin{align*}
&\frac{1-\underline{\alpha}}{2} |u_t(t)|_{L_2(\Omega)}^2 + \frac{c^2}{2} |\nabla u(t)|_{L_2(\Omega)}^2
+ \frac{c^2 \vep}{p+1} |\nabla u(t)|_{L_{p+1}(\Omega)}^{p+1} + \hat{b} \|\nabla u_t\|_{L_2(0,t;L_2(\Omega))} \\
&\leq \frac{1+\bar{\alpha}}{2} |u_1|_{L_2(\Omega)}^2 + \frac{c^2}{2} |\nabla u_0|_{L_2(\Omega)}^2
+ \frac{c^2 \vep}{p+1} |\nabla u_0|_{L_{p+1}(\Omega)}^{p+1}
\end{align*}
which, together with the initial bound \eqref{W1initialbound}, leads to the estimates
\begin{equation*}
\begin{aligned}
\|u_t\|_{L_\infty(0,T;L_2(\Omega))}^2 \leq \frac{2}{1-\underline{\alpha}} L\,  \kappa_T^2,\\
\|\nabla u_t\|_{L_2(0,T;L_2(\Omega))}^2 \leq \frac{1}{\hat{b}} L\,  \kappa_T^2,\\
\|\nabla u\|_{L_\infty(0,T;L_{p+1}(\Omega))}^{p+1} \leq \frac{p+1}{c^2 \vep} L\, \kappa_T^2,
\end{aligned}
\end{equation*}
where $L:=\frac{1+\bar{\alpha}}{2} + \frac{c^2}{2} + \frac{c^2 \vep}{p+1}$.
Choosing $\kappa_T$ so small that
\begin{equation*}
\frac{2}{1-\underline{\alpha}} L\, \kappa_T^2 \leq \bar{m}^2, \quad
\frac{1}{\hat{b}} L\,  \kappa_T^2\leq \bar{m}^2, \quad
\frac{p+1}{c^2 \vep} L\, \kappa_T^2 \leq \bar{M}^{p+1}
\end{equation*}
implies $\|u_t\|_{L_\infty(0,T;L_2(\Omega))} \leq \bar{m}$,
$\|\nabla u_t\|_{L_2(0,T;L_2(\Omega))} \leq \bar{m}$,
$\|\nabla u\|_{L_\infty(0,T;L_{p+1}(\Omega))} \leq \bar{M}$
and thus $u\in\mathcal{W}$ which proves that $\mathcal{T}: \mathcal{W} \to \mathcal{W}$ is a self-mapping. The closedness of $\mathcal{W}$ is trivial.
Existence of solutions can be obtained by a compactness argument: Since $\mathcal{W}$ is bounded in the dual of a
separable Banach space, it is w$\ast$-compact. Hence the sequence of fixed point iterates $u^n$ defined by
$\mathcal{T}u^n = \mathcal{T} u^{n-1}$ where $u_0$ is chosen compatible with initial and boundary conditions has a
w$\ast$-convergent subsequence whose w$\ast$-limit $\bar{u}$ lies in $\mathcal{W}$. Furthermore, as
by
\begin{equation*}
(1-2ku^{n-1})u^n_{tt}
-c^2\,\text{div}\Bigl(\nabla u^n+\vep|\nabla u^n|^{\pp-1}\nabla u^n\Bigr)-b\Delta u^n_t=2k u^{n-1}_t u^n_t
\end{equation*}
and integration by parts with respect to time we have
\begin{align*}
&\int_0^T \int_\Omega \big\{ -\bar{u}_t (( 1-2k \bar{u}) \phi )_t + [ c^2 ( \nabla \bar{u} + \varepsilon  | \nabla \bar{u} |^{p-1} \nabla \bar{u})
+ b \nabla \bar{u}_t ] \nabla \phi - 2k(\bar{u}_t )^2 \phi \big\} dx \, ds \\
&= \int_0^T \int_\Omega \Big\{ - (\bar{u}- u^n)_t ((1-2k\bar{u})\phi)_t + 2k u_t^n ((\bar{u}- u^{n-1} )\phi)_t \\
&- 2k(\bar{u}_t-u_t^n)\bar{u}_t\phi - 2k(\bar{u}_t-u_t^{n-1})u^n_t\phi + [c^2 \nabla(\bar{u}-u^n)+b\nabla (\bar{u}-u^n)_t]\nabla\phi\\
&+c^2 \varepsilon \int_0^1 \tilde{w}^\sigma [ |\nabla(u^n+\sigma \hat{u})|^2 \nabla \hat{u} +(p-1)(\nabla(u^n+\sigma \hat{u}) \nabla\hat{u})\nabla(u^n+\sigma\hat{u}) ] d\sigma \nabla\phi \Big\} \,dx\, ds\\
& \to 0 \text{ as } k \rightarrow \infty
\end{align*}
for any $\phi \in C_0^\infty((0,T)\times \Omega)$ where $\tilde{w}^\sigma(x,t)=|\nabla(u^n+\sigma\hat{u})|^{p-3}$, $\hat{u}=\bar{u}-u^n$, the w$\ast$-limit $\bar{u}$ satisfies the PDE in a weak sense which completes
the proof of existence.
\end{proof}

\begin{remark}
In Theorem \ref{th:W1locex}
but contractivity of $\mathcal{T}$ could not be proved, similarly to Sections \ref{secWpot_viscosity} and \ref{secWelastic_viscosity}; see also Remark \ref{rem:Wpot_viscosity_uniqueness}.
Suppose $v^i \in\mathcal{W}$, $u^i = \mathcal{T}v^i$, $i=1,2$ solve \eqref{Wpress_pLaplace}. Then, similarly to \eqref{W2_diff} we obtain
\begin{equation}\label{W1_diff}
\begin{cases}
(1-2kv^1)\hat{u}_{tt}-c^2\Delta \hat{u}
-c^2 \vep \int_0^1\,\text{div}\Bigl( \tilde{w}^\sigma
\bigl[|\nabla (u^2+\sigma \hat{u})|^2  \nabla \hat{u}\\
\quad+(p-1)(\nabla (u^2+\sigma \hat{u}) \nabla \hat{u})
\nabla(u^2+\sigma \hat{u}) \bigr]\,  \Bigr)d\sigma\\
\quad-b\Delta \hat{u}_t -2k v^1_t \hat{u}_t \ = \ 2k(\hat{v}u^2_{tt}+\hat{v}_tu^2_t)
\\
(\hat{u},\hat{u}_t)|_{t=0}=(0,0)
\\
\hat{u}|_{\partial \Omega} =0
.\\
\end{cases}
\end{equation}
for $\hat{u}=u^1-u^2$, $\hat{v}=v^1-v^2$, with $\tilde{w}^\sigma(x,t)= |\nabla(u^2+\sigma \hat{u})(x,t)|^{p-3}$. Multiplication of
\eqref{W1_diff} with $\hat{u}$ as well with $\hat{u}_t$ leads to problems with the first or fourth term, respectively.
Therefore, uniqueness of solutions of solutions remains open.
\end{remark}

For this equation we even have global existence and exponential decay of the energy
\[
E[u](t)=\frac12|\sqrt{1-2ku}u_t|_{L_2(\Omega)}^2+\frac{c^2}{2}|\nabla u|_{L_2(\Omega)}^2+\frac{c^2\vep}{\pp+1}|\nabla u|_{L_{\pp+1}(\Omega)}^{\pp+1}.
%
\]
\begin{theorem}\label{th:W1globex}
Let $c^2,b,\vep>0$, $k\in\R$, $\pp>d-1$.\\
There exist $\kappa,\bar{\kappa}>0$ such that for all $(u_0,u_1)\in W_0^{1,\pp+1}(\Omega)\times L_2(\Omega)$ with
\[
E[u](0)=\frac12|\sqrt{1-2ku_0}u_1|_{L_2(\Omega)}^2+\frac{c^2}{2}|\nabla u_0|_{L_2(\Omega)}^2+\frac{c^2\vep}{\pp+1}|\nabla u_0|_{L_{\pp+1}(\Omega)}^{\pp+1}\leq \kappa^2
\]
there exists a weak solution $u$ of \eqref{Wpress_pLaplace} for all times which satisfies
\[
E[u](t)\leq \bar{\kappa}^2.
\]
Moreover, there exists $\omega>0$ such that the exponential decay estimate
\[
E[u](t)\leq E[u](0)\exp(-\omega t)
\]
holds.
\end{theorem}

\begin{proof}
The crucial energy estimate is obtained by multiplication of \eqref{Wpress_pLaplace} with $u_t$ (using $(1-2ku)u_tu_{tt}=\frac12 \frac{d}{dt}[(1-2ku)(u_t)^2+k(u_t)^3$)
\begin{align*}
&\Bigl[\frac12|\sqrt{1-2ku}u_t|_{L_2(\Omega)}^2
+\frac{c^2}{2}|\nabla u|_{L_2(\Omega)}^2
+\frac{c^2\vep}{\pp+1}|\nabla u|_{L_{\pp+1}(\Omega)}^{\pp+1}
\Bigr]_0^t
+b\int_0^t|\nabla u_t(s)|_{L_2(\Omega)}^2\,ds\\
&=k\int_0^t \int_\Omega (u_t(s))^3\,dx\,ds
\leq k\sup_{s\in(0,t)} |u_t(s)|_{L_2(\Omega)} \ \int_0^t |u_t|_{L_4(\Omega)}^2 \, ds\\
&\leq k {C_{H_0^1,L_4}^\Omega}^2 \sup_{s\in(0,t)} |u_t(s)|_{L_2(\Omega)} \ \int_0^t |\nabla u_t|_{L_2(\Omega)}^2 \, ds
\end{align*}

as long as $u$ exists and is pointwise bounded away from $\frac{1}{2k}$
\begin{equation*}
-1 < -\underline{\alpha}\leq -2k u(s,x)\leq\overline{\alpha} \quad \forall s\in(0,t)\,, \ x\in\Omega
\end{equation*}
which is true for sufficiently short time according to Theorem \ref{th:W1locex}.
Hence, if
\[
|u_t(0)|_{L_2}^2\leq \frac{2}{1-\underline{\alpha}} E[u](0)\leq \frac{2}{1-\underline{\alpha}} \kappa^2<  \Big(\frac{b}{k {(C_{H_0^1,L_4}^\Omega)}^2}\Big)^2,
\]
the energy decreases monotonically with time,
\begin{equation}\label{enestW1_1}
E[u](t)+\tilde{b}\int_0^t|\nabla u_t(s)|_{L_2(\Omega)}^2\, ds\leq E[u](0)
\end{equation}
with $\tilde{b}=b-\sqrt{\frac{2}{1-\underline{\alpha}}}\kappa k ({C_{H_0^1,L_4}^\Omega})^2 >0$ and global existence can be concluded.
Equipartition of energy, which we get by multiplication of \eqref{Wpress_pLaplace} with $u$ and
integration over space and time, here reads as
\begin{eqnarray*}
\int_0^t\Bigl\{-|\sqrt{1-4ku}u_t|_{L_2(\Omega)}^2
+c^2|\nabla u|_{L_2(\Omega)}^2
+c^2\vep|\nabla u|_{L_{\pp+1}(\Omega)}^{\pp+1}
\Bigr\}\,ds\\
+b\left[|\nabla u|_{L_2(\Omega)}^2\right]_0^t
=-\left[ \int_\Omega (1-2ku)u u_t\, dx\right]_0^t,
\end{eqnarray*}
hence
\begin{equation}
\label{equipartW1}
\begin{aligned}
&\int_0^t|u_t(s)|_{L_2(\Omega)}^2\, ds
\geq \frac{1}{1+2\overline{\alpha}}\int_0^t\Bigl\{
c^2|\nabla u|_{L_2(\Omega)}^2
+c^2\vep|\nabla u|_{L_{\pp+1}(\Omega)}^{\pp+1}
\Bigr\}\,ds\\
&+\frac{b}{1+2\overline{\alpha}}\Bigl(|\nabla u(t)|_{L_2(\Omega)}^2-|\nabla u(0)|_{L_2(\Omega)}^2\Bigr)-\frac{1+\overline{\alpha}}{1+2\overline{\alpha}}\Bigl(C_{PF}^2|\nabla u(t)|_{L_2(\Omega)}^2\\
&+ |u_t(t)|_{L_2(\Omega)}^2+C_{PF}^2|\nabla u(0)|_{L_2(\Omega)}^2+ |u_t(0)|_{L_2(\Omega)}^2\Bigr).
\end{aligned}
\end{equation}
Thus we can split the term $b\int_0^t|\nabla u_t(s)|_{L_2(\Omega)}^2\,ds$ in \eqref{enestW1_1} as follows
\begin{align*}
\lefteqn{b\int_0^t|\nabla u_t(s)|_{L_2(\Omega)}^2\,ds}\\
&\geq \frac{b}{C_{PF}^2} \int_0^t\Bigl\{(1-\lambda)|u_t(s)|_{L_2(\Omega)}^2
+\frac{\lambda c^2}{1+2\overline{\alpha}}|\nabla u(s)|_{L_2(\Omega)}^2
+\frac{\lambda c^2\vep}{1+2\overline{\alpha}}|\nabla u(s)|_{L_{\pp+1}(\Omega)}^{\pp+1}
\Bigr\}\,ds
\nonumber\\
&+\frac{\lambda b^2}{C_{PF}^2(1+2\overline{\alpha})}
\Bigl(|\nabla u(t)|_{L_2(\Omega)}^2-|\nabla u(0)|_{L_2(\Omega)}^2\Bigr)
\nonumber\\
&-\frac{\lambda b}{2C_{PF}^2}\frac{1+\overline{\alpha}}{1+2\overline{\alpha}}
\Bigl(C_{PF}^2|\nabla u(t)|_{L_2(\Omega)}^2+ |u_t(t)|_{L_2(\Omega)}^2
+C_{PF}^2|\nabla u(0)|_{L_2(\Omega)}^2+ |u_t(0)|_{L_2(\Omega)}^2\Bigr)
\end{align*}
which inserted into \eqref{enestW1_1} with $\lambda\in(0,1)$ sufficiently small implies
\[
E[u](t)+c\int_0^t E[u](s)\, ds\leq C E[u](0)
\]
for some constants $c,C>0$ and all $t>0$.
This relation by a standard argument implies exponential decay of $E[u](t)$.
\end{proof}

\section{Conclusions and Remarks}

The introduction of nonlinear strong damping in equations of nonlinear acoustics allows us to prove existence of solutions with weaker regularity; in particular, this enables us to show well-posedness of solutions to coupled acoustic-acoustic and acoustic-elastic problems.
However, the nonlinear strong damping also introduces additional challenges to the analysis: due to the relatively
high order of differentiation they contain, these terms only allow to derive energy estimates for certain low order multipliers.
For this reason, for some of the equations under consideration, uniqueness of solutions remains open.

The presence of the linear strong damping term $-\Delta u_t$ would seem to imply that the use of a
nonlinear strong damping  $-\text{div}\,(|\nabla u_t|^{q-1}\nabla u_t)$ of viscosity type is more natural, it turns out that a $p$-Laplace damping
term $-\text{div}\,(|\nabla u|^{q-1}\nabla u)$ yields some nicer mathematical properties such as global in time existence and
exponential decay; uniqueness, however, remains an open problem for this formulation.
Here we have only investigated the $p$-Laplace damping for the acoustic pressure formulation;
we do expect that global existence and exponential decay will also carry over to the velocity potential formulation
\eqref{Wpot_viscosity} and to the coupled problems \eqref{Wacoustic_viscosity}, \eqref{Welastic_viscosity}
upon replacement of viscosity by $p$-Laplace damping. This setting together with a choice of different boundary conditions
(e.g., practically relevant Neumann as well as absorbing boundary conditions) will be subject of future research.
\bigskip

{\bf Acknowledgments.}
The authors gratefully acknowledge the referee's careful reading of the manuscript and many fruitful comments which led to an improved version of this paper.

R.B. and B.K. gratefully acknowledge financial support of their research by the FWF (Austrian Science Fund): P24970. The work of P.R. was supported by the NSF Grant DMS 0908435.

\end{document}